\documentclass[11pt,reqno]{amsart}
\usepackage{fullpage}
\usepackage[utf8]{inputenc}
\usepackage[colorlinks,
            linkcolor=red,
            anchorcolor=red,
            citecolor=blue
            ]{hyperref}

\usepackage{graphicx, amsmath, fullpage, amssymb, amsthm, amsfonts, color, algorithm,mathtools}
\mathtoolsset{showonlyrefs}
\usepackage{enumerate}
\usepackage{cases}
\newtheorem{theorem}{Theorem}[section]
\newtheorem{lemma}[theorem]{Lemma}
\usepackage{comment}

\theoremstyle{definition}
\newtheorem{definition}[theorem]{Definition}

\newtheorem{cor}[theorem]{Corollary}
\newtheorem{remark}[theorem]{Remark}

\title{Non-backtracking eigenvalues and eigenvectors of random regular graphs and hypergraphs}
\author{Xiangyi Zhu}
\address{Department of Applied Mathematics and Statistics, Johns Hopkins University,
Baltimore, MD 21218}
\email{xzhu96@jh.edu}

\author{Yizhe Zhu}
\address{Department of Mathematics, University of Southern California, Los Angeles, CA 90007}
\email{yizhezhu@usc.edu}
\date{\today}

\begin{document}

\maketitle

\begin{abstract}
The non-backtracking operator of a graph is a powerful tool in spectral graph theory and random matrix theory. Most existing results for the non-backtracking operator of a random graph concern only eigenvalues or top eigenvectors. In this paper, we take the first step in analyzing its bulk eigenvector behaviors. We demonstrate that for the non-backtracking operator $B$ of a random $d$-regular graph, its eigenvectors corresponding to nontrivial eigenvalues are completely delocalized with high probability. Additionally, we show complete delocalization for a reduced $2n \times 2n$ non-backtracking matrix $\tilde{B}$. By projecting all eigenvalues of $\tilde{B}$ onto the real line, we obtain an empirical measure that converges weakly in probability to the Kesten-McKay law for fixed $d\geq 3$ and to a semicircle law as $d \to\infty$  with $n \to\infty$.
We extend our analysis to random regular hypergraphs, including the limiting measure of the real part of the spectrum for $\tilde B$, $\ell_{\infty}$-norm bounds for the eigenvectors of $\tilde B$ and $B$, and a deterministic relation between eigenvectors of $B$ and the eigenvectors of the adjacency matrix. 

As an application, we analyze the non-backtracking spectrum of the $(d_1,d_2)$-regular stochastic block model (RSBM) and provide a spectral method based on eigenvectors of $\tilde B$ to recover the community structure exactly. We also show that there exists an isolated real eigenvalue with an informative eigenvector inside the circle of radius $\sqrt{d_1+d_2-1}$ in the spectrum of $B$, analogous to the ``eigenvalue insider" phenomenon previously observed for the Erd\H{o}s-R\'{e}nyi stochastic block model.
\end{abstract}

\section{Introduction}

\subsection{Non-backtracking operators of random graphs}
The non-backtracking operator is an important object in the study of spectral graph theory \cite{terras2010zeta,angel2015non,bass_iharaselberg_1992,kempton2016non,glover2021some,mulas2024there,jost2023spectral,torres2021nonbacktracking}. It has recently been used as a powerful tool for studying random matrices \cite{bordenave2018nonbacktracking,bordenave2020new,benaych2020spectral} and for designing efficient algorithms in community detection and matrix completion \cite{krzakala2013spectral,bordenave2018nonbacktracking,stephan2020non,bordenave2020detection,stephan2022sparse,stephan2023non}. Many recent results on the spectrum of the non-backtracking operator have been established in various graph models, such as Erdős-Rényi graphs \cite{wang2017limiting}, stochastic block models \cite{bordenave2018nonbacktracking,coste2021eigenvalues}, inhomogeneous random graphs \cite{benaych2020spectral,dumitriu2022extreme}, random regular graphs \cite{bordenave2020new}, and bipartite biregular graphs \cite{brito2021spectral}. Very recently, some of these results have also been extended to hypergraphs \cite{dumitriu2021spectra,stephan2022sparse,chodrow2023nonbacktracking}. However, so far, these results mainly focus on top eigenvalues and eigenvectors and global spectral distribution, while not much is known about bulk eigenvector behavior.  

In this paper, we take a first step towards understanding the bulk eigenvectors of the non-backtracking operator in random graphs. Eigenvector delocalization is an important topic in the study of random matrices, demonstrating that many matrix models exhibit behavior similar to their Gaussian analog, a phenomenon known as universality. A unit eigenvector is considered delocalized if its infinity norm is close to that of a uniformly distributed random vector on the unit sphere up to a polylog factor. Eigenvector delocalization has been shown for many random matrix models, including non-Hermitian ones. However, for sparse and non-Hermitian models, the literature includes only a few results, such as those on eigenvector delocalization in random regular digraphs \cite{litvak2019structure} and Erdős-Rényi digraphs \cite{he2023edge} with growing average degrees.

The non-backtracking operator \( B \) for a random \( d \)-regular graph is a sparse and non-Hermitian operator, where each row and column of \( B \) has {exactly} \( (d-1) \) nonzero entries. It can also be viewed as the adjacency matrix of a random \( (d-1) \)-regular digraph with \( (nd) \) vertices. However, not all \( (d-1) \)-regular digraphs can be obtained in this way (see \cite[Theorem 2.3]{mulas2024there}), and the eigenvalue distribution of $B$ is very different from the conjectured oriented Kesten-McKay law for a uniformly chosen random \( (d-1) \)-regular digraph \cite{bordenave2012around}. A simulation of the non-backtracking spectrum of a random $d$-regular graph is shown in Figure \ref{fig:eigenvalues}.

 \begin{figure} 
\centering\includegraphics[width=1\linewidth,height=0.25\textheight]{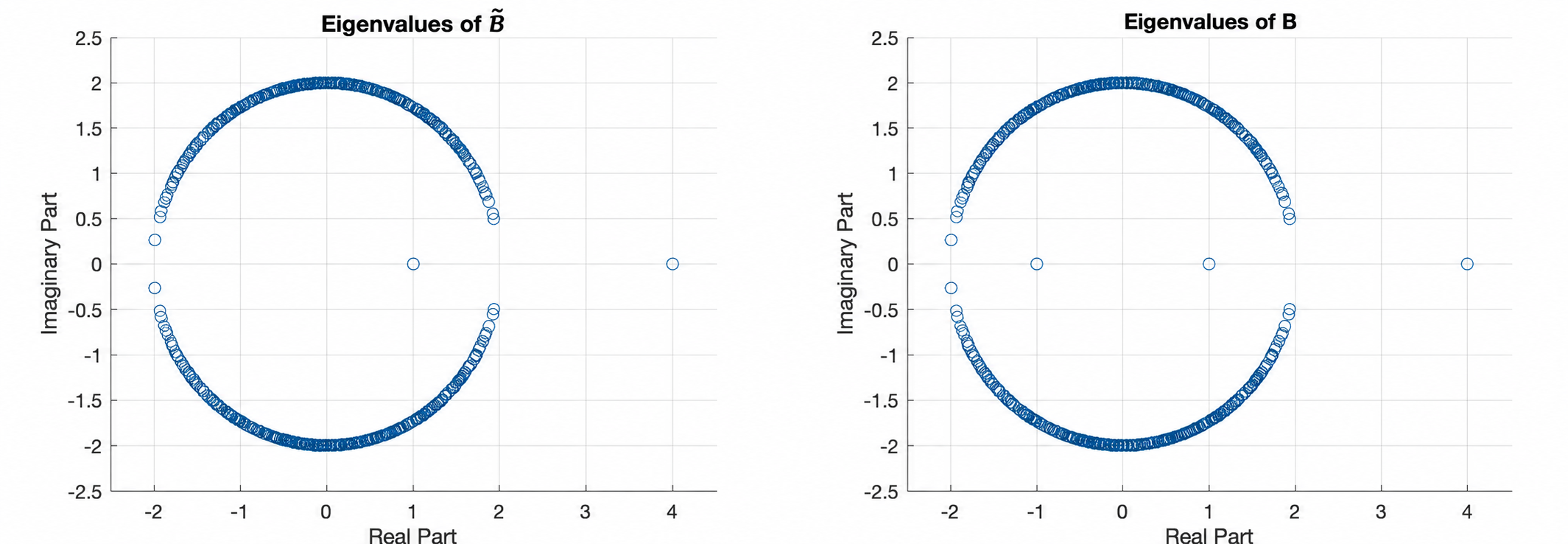}
    \vspace{-15pt}
    \caption{Simulation for eigenvalues of  $ \tilde B$ and $B$ for a random $5$-regular graph with 200 vertices.}
    \label{fig:eigenvalues}
\end{figure}
In Theorem \ref{thm:B}, we show that with high probability, all eigenvectors corresponding to nontrivial non-backtracking eigenvalues are completely delocalized. The main idea behind the proof is to use the delocalization \cite{huang2021spectrum} and spectral gap \cite{huang2021spectrum,bordenave2020new} results for the adjacency matrix of a random \( d \)-regular graph to understand eigenvector delocalization for the non-backtracking matrix with the help of the Ihara-Bass formula \cite{bass_iharaselberg_1992}. The algebraic structure of the non-backtracking operator enables the translation of eigenvalue and eigenvector information from the corresponding adjacency matrix. However, this precise algebraic connection between these two operators is only available for regular graphs. Exploring beyond \( d \)-regular graphs is an interesting future direction. {Nevertheless, to the best of our knowledge, this is the first random directed graph model with bounded degree for which eigenvector delocalization has been established.}


In the course of proving non-backtracking eigenvector delocalization, we show in Theorem \ref{thm:real} the convergence in probability to the Kesten-McKay law and semicircle law for fixed \( d \) and growing \( d \) when projecting the eigenvalues of the non-backtracking operator onto the real line. From the Ihara-Bass formula, another \( 2n \times 2n \) non-Hermitian block matrix \( \tilde B \) closely related to \( B \) emerges, which we call a reduced non-backtracking operator \cite{angel2015non,wang2017limiting,coste2021eigenvalues,stephan2019robustness}. We also demonstrate that the eigenvectors of \( \tilde B \) are completely delocalized with high probability in Theorem \ref{thm:tildeB}.


\subsection{Extension to regular hypergraphs}
 The spectral theory for hypergraphs has drawn considerable interest due to its applications in number theory, community detection, and network analysis \cite{feng1996spectra,li2004ramanujan,sole1996spectra,storm2006zeta}. In \cite{angelini2015spectral,stephan2022sparse,chodrow2023nonbacktracking},  spectral algorithms based on the non-backtracking eigenvectors were studied for community detection problems in hypergraph stochastic block models. Several results we obtained for random regular graphs can be generalized to random regular hypergraphs, which is a popular hypergraph model studied in combinatorics, statistical physics, and computer science \cite{cooper1996perfect,dumitriu2021spectra,greenhill2022spanning,chang2023upper,sen2018optimization}.
 
 In Theorem~\ref{thm:bi_real}, we show that by projecting the eigenvalues of a non-backtracking operator $B$ for random regular hypergraphs, we can obtain different limiting measures depending on the regime of $d,k$. We provide an exact spectral relation between a reduced non-backtracking operator $\tilde B$ and the adjacency matrix $A$ of a regular hypergraph in Lemma~ \ref{lem:bi_eigenvalues}.  A precise relation between eigenvectors of $B$ and $A$ is given in  Lemma~\ref{lem:eigenvector_hypergraph}, which generalizes the result for regular graphs in \cite{lubetzky2016cutoff}. These relations are used to show $\ell_{\infty}$-norm {bounds} on eigenvectors of $\tilde B$ and $B$ in Theorem~\ref{thm:bi_tildeB}.

\subsection{Application in a regular stochastic block model}
Community detection in the stochastic block model is an important topic in statistical physics, machine learning, and network science \cite{abbe2018community}. Using the non-backtracking operator to detect the community structure is proven to achieve the information-theoretical threshold in many settings \cite{bordenave2018nonbacktracking,stephan2022sparse}. However, the spectrum of the non-backtracking operator for stochastic block models is not fully understood. An intriguing behavior of the {spectrum} for the non-backtracking matrix $B$ in the Erd\H{o}s-R\'{e}nyi stochastic block model with constant expected degree was observed in \cite{dall2019revisiting}, {namely}, that there exists a real eigenvalue inside the bulk spectrum of $B$ close to the ratio of the top two eigenvalues of $B$. This was justified in \cite{coste2021eigenvalues} in the dense regime when the expected degree is $\omega(\log n)$. We study an analog of the stochastic block model in the random regular graph setting, which is called the \textit{regular stochastic block model}  (RSBM) \cite{brito2016recovery,barucca2017spectral,newman2014equitable} studied in the literature, which exhibits different behaviors from the Erd\H{o}s-R\'{e}nyi {stochastic block model} (SBM).  The RSBM is constructed by combining a random $d_2$-regular bipartite graph and a random  $d_1$-regular graph; see Definition~\ref{def:RSBM} for a precise description.

We confirm the \textit{eigenvalue insider} phenomenon in the RSBM in Theorem~\ref{thm:RSBM} and Corollary~\ref{cor:recovery}: When $(d_1-d_2)^2>4(d_1+d_2-1)$, a real eigenvalue of $\tilde B$ exists inside the circle of radius $\sqrt{d_1+d_2-1}$, and the corresponding eigenvector reveals the community structure exactly. {For the  two-community Erd\H{o}s-R\'{e}nyi SBM, weak recovery is information-theoretically impossible below the Kesten--Stigum threshold \cite{mossel2015reconstruction}. By contrast, in the RSBM, the community assignment vector $\sigma$ is an exact adjacency eigenvector with eigenvalue $d_1-d_2$. Based on this algebraic structure and numerical evidence, Barucca
\cite{barucca2017spectral} conjectured that the planted partition admits efficient exact recovery whenever $d_1>d_2>0$; to the best of our knowledge, this conjecture remains open in the full parameter regime. 
}

\subsection*{Organization of the paper}
The rest of the paper is organized as follows. In Section \ref{sec:prelim}, we provide some basic {notation} and definitions for regular graphs, the non-backtracking operator, and hypergraphs. In Section \ref{sec:main}, we state the main results for random regular graphs, random regular hypergraphs, and the regular stochastic block model, respectively. Sections \ref{sec:thm1} to \ref{sec:proof_RSBM} contain all the proofs. In Section \ref{sec:ER}, we discuss the open problem of generalizing the results to Erdős-Rényi graphs.


\section{Preliminaries}\label{sec:prelim}
\subsection{Regular graphs}
Let $G=(V,E)$ be a graph. $G$ is $d$-regular of size $n$ if each vertex has degree $d$ and $|V|=n$. For given $n$ and $d$, we say $G$ is a \textit{random $d$-regular graph} if it is uniformly chosen from all $d$-regular graphs with $n$ vertices.

The $(i,j)$-th entry of the \textit{adjacency matrix} $A$ of a graph $G$ is defined as 
\[ A_{ij}=\begin{cases} 
    1  &  \text{if } \{i,j\}\in E,\\
    0 & \text{otherwise.}
\end{cases}\]
The degree matrix $D$ of a graph $G$ is a diagonal matrix where $D_{ii}=\sum_{j\in V} A_{ij}$.
Define the oriented edge set $\vec{E}$ for $G$ as 
\begin{align} \label{eq:def_oriented}
\vec{E}= \{ (i,j): \{ i,j\} \in E\}.
\end{align}
Each edge yields two oriented edges; therefore, $|\vec{E}| =2|E|$. 

\begin{definition}[Non-backtracking operator]
  The \textit{non-backtracking operator} $B$ of $G$ is a non-Hermitian operator of size $|\vec E| \times |\vec E|$. For any $(u,v), (x,y)\in \vec{E}$, $B$ is defined as $$
B_{(u,v),(x,y)} = \begin{cases}
     1 &\text{$v=x$, $u \neq y$}, \\
     0 &\text{otherwise.}
\end{cases}
$$  
\end{definition}

In particular, for a $d$-regular graph with $n$ vertices, the corresponding $B$ is of size $nd\times nd$.  A useful identity we will use in this paper is the following Ihara-Bass formula \cite{bass_iharaselberg_1992}.
\begin{lemma}[Ihara-Bass formula]\label{lem:Ihara-Bass} For any graph $G=(V,E)$, and any $z\in \mathbb C$, the following identity holds:
    \begin{align}\label{eq:Ihara}
 \det(B-zI)=(z^2-1)^{|E|-n} \det (z^2I -zA+D-I).       
    \end{align}
\end{lemma}
Define a block matrix
\begin{align}\label{eq:tildeB}
    \tilde B=\begin{bmatrix}
0 &  D-I\\
-I  & A
\end{bmatrix} \in \mathbb R^{2n\times 2n}.
\end{align}
Then from \eqref{eq:Ihara}, we have 
\begin{align}\label{eq:det_tildB}
     \det(B-zI)=(z^2-1)^{|E|-n}\det (\tilde B-zI).
\end{align}
The identity \eqref{eq:det_tildB} implies that $B$ and $\tilde B$ share the same spectrum, up to the multiplicity of trivial eigenvalues $\pm 1$. We also call $\tilde B$ the \textit{reduced non-backtracking matrix} of $G$.  When $G$ is a $d$-regular graph with $n$ vertices, \eqref{eq:Ihara} can be further simplified to
 \begin{align}\label{eq:Ihara_regular}
      \det(B-zI)=(z^2-1)^{nd/2-n} \det (z^2I -zA+(d-1) I).       
 \end{align}

\subsection{Regular hypergraphs}
In this section, we include some standard definitions of hypergraphs.
\begin{definition}[Hypergraph]
    A \textit{hypergraph} $H$ consists of a set $V$ of vertices and a set $E$ of hyperedges such that each hyperedge is a nonempty set of $V$.  A hypergraph $H$ is \textit{$k$-uniform} for an integer $k\geq 2$ if every hyperedge $e\in E$ contains exactly $k$ vertices. The \textit{degree} of $i$, denoted   $\deg(i)$, is the number of all hyperedges incident to $i$. 
\end{definition}
	 A hypergraph is \textit{$d$-regular} if all of its vertices have degree $d$.
A hypergraph is \textit{$(d,k)$-regular} if it is both $d$-regular and $k$-uniform. See Figure~\ref{fig:hypergraph} for an example.
\begin{figure}
    \centering
\includegraphics[width=0.25\linewidth]{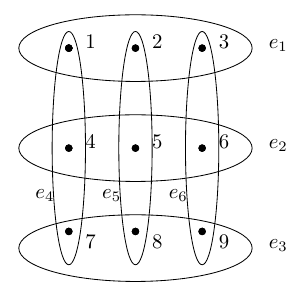}
    \caption{A $(2,3)$-regular hypergraph with $9$ vertices}
    \label{fig:hypergraph}
\end{figure}

\begin{definition}[Adjacency matrix of a hypergraph]
  For a hypergraph $H$ with $n$ vertices, we associate a $n\times n$ symmetric matrix $A$ called the \textit{adjacency matrix} of $H$. For $i\not=j$, we define $A_{ij}$ as the number of hyperedges containing both $i$ and $j$; we define $A_{ii}=0$ for all $1\leq i\leq n$. When the hypergraph is $2$-uniform, this is the definition for the adjacency matrix of a graph.  
\end{definition}

\section{Main results}\label{sec:main}
\subsection{Random regular graphs}\label{sec:regular}
   Let $\mu_1,\dots,\mu_{2n}$ be the eigenvalues of $\tilde B$. Define $x_i$ to be the real part of $\mu_i$. The empirical measure of $\{x_i, 1\leq i\leq 2n\}$ is denoted as 
\[ \mu=\frac{1}{2n} \sum_{i=1}^{2n} \delta_{x_i}.\]

The following theorem characterizes the eigenvalue distribution of $\tilde B$ after projecting all eigenvalues to the real line.
\begin{theorem}[Projecting eigenvalues of $\tilde B$ to the real line]\label{thm:real}
The following holds for the reduced non-backtracking matrix $\tilde B$ of a uniformly chosen random $d$-regular graph:
\begin{enumerate}
    \item When $d\geq 3$ is a fixed integer, $\mu$ converges weakly in probability to a rescaled Kesten-McKay distribution $\mu_{\mathrm{KM}}$ supported on $[-\sqrt{d-1},\sqrt{d-1}]$, where 
    \[ \mu_{\mathrm{KM}}(x)= \frac{2d\sqrt{(d-1)-x^2} }{\pi(d^2 - 4x^2)} ~ \mathbf{1} \left\{ |x|\leq \sqrt{d-1}\right\}.\] 
    \item When $d\to\infty$ as $n\to\infty$, the empirical measure of $\left\{\frac{2x_i}{\sqrt{d-1}}, i\in [2n]\right\}$ converges weakly in probability to a semicircle law $\mu_{\mathrm{SC}}$, where 
    \begin{align}\label{eq:semicircle}
    \mu_{\mathrm{SC}}(x)=\frac{1}{2\pi}\sqrt{4-x^2}~ \mathbf{1}\{ |x|\leq 2\}.
    \end{align}
\end{enumerate}
\end{theorem}

Figure \ref{fig:eigenvalues_projection} is a simulation for the projected eigenvalues on the real line of a random regular graph. 
Similar results have been established for Erd\H{o}s-R\'{e}nyi graphs $G(n,p)$ with $np=\omega(\log n)$ \cite{wang2017limiting}  and for stochastic block models \cite{coste2021eigenvalues} in the same regime.
\begin{figure} 
    \centering
    \includegraphics[width=\linewidth,height=0.22\textheight]{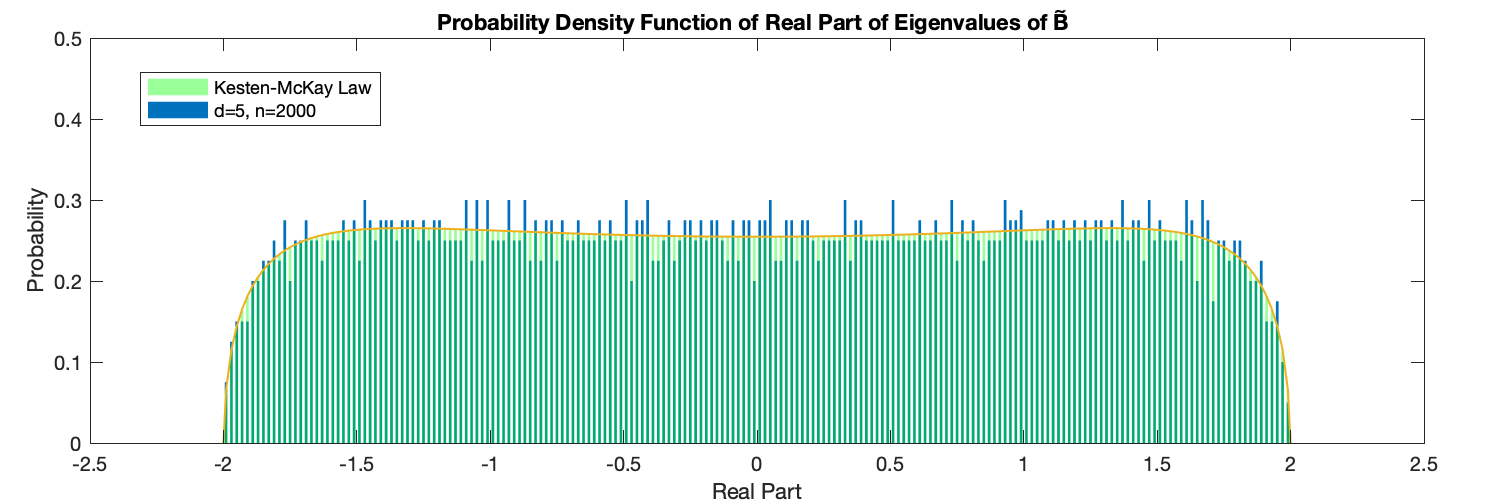}
    \vspace{-15pt}
    \caption{Simulation for the projected eigenvalues on the real line of a random regular graph.}
    \label{fig:eigenvalues_projection}
\end{figure}

We now move on to study the eigenvectors of $B$ and $\tilde B$.
Note that each unit eigenvector of $\tilde B$ is in $\mathbb C^{2n}$. We show they are completely delocalized based on the eigenvector delocalization results in \cite{huang2021spectrum} for the adjacency matrix of a random $d$-regular graph.
\begin{theorem}[Eigenvector delocalization for $\tilde B$] \label{thm:tildeB}
Let $d\geq 3$ be fixed, and $\tilde B$ be the reduced non-backtracking matrix of a random $d$-regular graph. Let $u_i $ for $  i\in [2n]$ be the $\ell_2$-normalized eigenvector associated with $\mu_i$. Then there exist absolute  constants $C_1, C_2>0$ such that with probability at least $1-n^{-C_1}$, for all $i\in [2n]$,
\begin{align}
    \| u_i\|_{\infty}\leq  \frac{\log^{C_2}(n)}{\sqrt n}.
\end{align}
\end{theorem}

By a different argument with the help of the Ihara-Bass formula, we can also show that unit eigenvectors of $B$ as a vector in $\mathbb C^{nd}$ are completely delocalized.

\begin{theorem}[Eigenvector delocalization for $B$]\label{thm:B}
Let $d\geq 3$ be fixed, and  $B$ be the non-backtracking matrix of a random $d$-regular graph. Let $\mu_i $ {for} $ i\in [2n]$ be the eigenvalues of $\tilde B$. Then  there exist absolute  constants $C_1, C_2>0$ such that with probability at least $1-n^{-C_1}$, all the  unit eigenvectors $w_i$ of $B$ associated with  $\mu_i\not=\pm 1$  satisfy
\begin{align}
    \| w_i\|_{\infty}\leq  \frac{\log^{C_2}(n)}{\sqrt {nd}}.
    \end{align}
\end{theorem}

\begin{figure} 
    \centering
    \includegraphics[width=\linewidth]{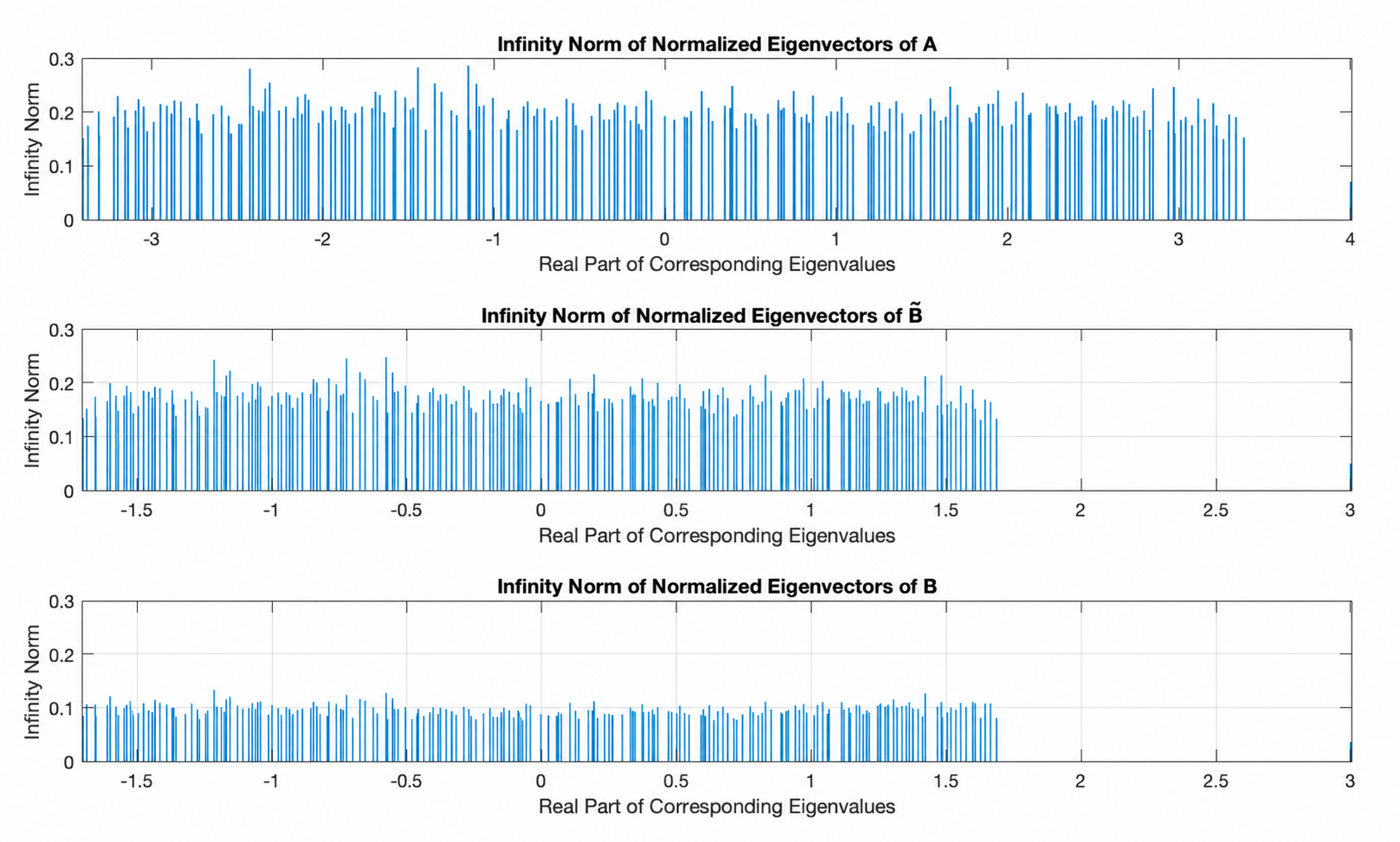}
    \vspace{-10pt}
    \caption{The infinity norms of  eigenvectors for $A$, $\tilde B$, $B$ of a $4$-regular random graph with $n=200$. }
    \label{fig:delocalization}
\end{figure}

The $\ell_{\infty}$-norms of eigenvectors of $B$ and $\tilde B$ are illustrated in Figure \ref{fig:delocalization}. Similar results for non-backtracking eigenvectors of random $d$-regular {graphs} for growing $d$ can be proved in the same way by using the eigenvector delocalization bounds for the adjacency matrix in \cite{dumitriu2012sparse,tran2013sparse,bauerschmidt2017local,bauerschmidt2019local}.

\begin{remark}[Eigenvectors  corresponding to  eigenvalues $\pm 1$ in Equation \eqref{eq:Ihara_regular}]
In Equation \eqref{eq:Ihara_regular}, we see there are trivial eigenvalues $\pm 1$ of $B$  with an extra multiplicity $\frac{nd}{2}-n$ which are not given by eigenvalues of $\tilde B$. The structure of the eigenspaces of $\pm 1$ was discussed in \cite[Proof of Proposition 3.1]{lubetzky2016cutoff}. 
\end{remark}

\subsection{Random regular hypergraphs}

\begin{definition}[Non-backtracking operator of a  hypergraph]
   For a hypergraph $H=(V,E)$, its \textit{non-backtracking operator} $B$ is a square matrix indexed by oriented hyperedges 
\[\vec{E}=\{(i,e): i\in V, e\in E, i\in e\}\] with entries given by 
$$B_{(i, e), (j,f)}=\begin{cases}
	1 & \text{if $j\in e\setminus\{i\}, f\not=e$,}\\
	0 & \text{otherwise,}
\end{cases}
$$
for any oriented hyperedges $(i, e),(j, f)\in \vec{E}$.
\end{definition} 
 This is a generalization of the graph non-backtracking operators to hypergraphs. 
For a $k$-uniform hypergraph $H=(V,E)$, define the $2n \times 2n$ \textit{reduced non-backtracking matrix} $\tilde B$ as
\begin{equation}\label{eq:def_tilde_B}
\tilde B = \begin{pmatrix} 0 &  (D-I)\\  -(k-1)I & A-(k-2)I
\end{pmatrix},
\end{equation}
where $A$ is the adjacency matrix of $H$, and $D$ is the diagonal \textit{degree matrix} with
\begin{align} \label{eq:def_degree_matrix}
D_{ii} = \#\{e\in E:i\in e \}.
\end{align}

The following Ihara-Bass formula for hypergraphs was proved in \cite{stephan2022sparse}. Lemma~\ref{lem:Ihara-Bass_H} shows that the spectrum of $\tilde B$ is identical to that of $B$, except for possible trivial eigenvalues at ${1}$ and $-(k-1)$. 
\begin{lemma}[Lemma 1 in \cite{stephan2022sparse}]\label{lem:Ihara-Bass_H}
Let $H=(V,E)$ be a $k$-uniform hypergraph. The following identity holds for any $z\in \mathbb C$:
\begin{align*}
    &\det(B-zI)\\
    &=(z-1)^{(k-1)|E|-n}(z+(k-1))^{|E|-n}\det \left( z^2+(k-2)z-zA+(k-1)(D-I)\right)\\
    &=(z-1)^{(k-1)|E|-n}(z+(k-1))^{|E|-n}\det(\tilde B-zI).
\end{align*}
\end{lemma}

We obtain the following limiting distributions for the real part of eigenvalues in $\tilde B$ for a uniformly chosen random $(d,k)$-regular hypergraph.

\begin{theorem} [Projecting eigenvalues of $\tilde B$ to the real line]\label{thm:bi_real}  Let $\tilde B$ be the reduced non-backtracking operator of a random  $(d,k)$-regular hypergraph. Let $x_1,\dots,x_{2n}$ be the real part of the eigenvalues of $\tilde B$. Define $\mu$ to be the empirical measure of  $\left\{ \frac{{2}x_i}{\sqrt{(d-1)(k-1)}}\right\}_{i=1}^{2n}$ . Then the following holds:
\begin{enumerate}
    \item If $d,k$ are fixed, as $n\to\infty$, {we have} $\mu$ converges weakly in probability to a distribution supported on $[-2,2]$ whose density function is given by
   \begin{align}
	\mu_{d,k}(x)=\frac{1+\frac{k-1}{q}}{(1+\frac{1}{q}-\frac{x}{\sqrt{q}})(1+\frac{(k-1)^2}{q}+\frac{(k-1)x}{\sqrt{q}})\pi}\sqrt{1-\frac{x^2}{4}},
	\end{align}
	where $q=(k-1)(d-1)$.
    \item If $d/k\to \alpha$ as $n\to\infty$ and $d\leq \frac{n}{32}$, {we have} $\mu$ converges weakly in probability to a distribution supported on $[-2,2]$ with density function 
    \begin{align}
        \mu_{\alpha}(x)=\frac{\alpha}{(1+\alpha+\sqrt{\alpha}x)\pi}\sqrt{1-\frac{x^2}{4}}.
    \end{align}
    \item If $d\to\infty, d=o(n^{\varepsilon})$ for any $\varepsilon>0$ and $\frac{d}{k}\to\infty$, $\mu$ converges weakly in probability to the semicircle law given in \eqref{eq:semicircle}.
\end{enumerate}
\end{theorem}

We obtain the following $\ell_{\infty}$-norm bounds for eigenvectors of $\tilde B$ and $B$. 

\begin{theorem}[$\ell_{\infty}$-norm bound for eigenvectors of $\tilde B$ and $B$]\label{thm:bi_tildeB}
    Let $B$ and $\tilde B$ be the non-backtracking operator and reduced non-backtracking operator of a $(d,k)$-regular hypergraph, respectively. Let $d,k$ be fixed. Then the following holds:
 \begin{enumerate}
     \item Let $v_i$ be {a} unit  eigenvector of $A$ associated with the eigenvalue $\lambda_i$. Let $u_i,u_i'$ be unit eigenvectors of $\tilde B$ associated with two eigenvalues $\mu_i, \mu_i'$, which are given by the solutions of  
     \begin{align}\label{eq:quadtatic_mu}
     \mu^2 - (\lambda_i-k+2 )\mu + (d-1)(k-1) = 0.
     \end{align} Then  for all $i\in [n]$,
   \[  {\max}\{  \|u_i\|_{\infty}, \|u_i'\|_{\infty} \} \leq \|v_i\|_{\infty}.\]
     \item For a random $(d,k)$-regular {hypergraph}, let $w_i,w_i'$ be unit eigenvectors of $B$ associated with two eigenvalues $\mu_i, \mu_i'$ given by \eqref{eq:quadtatic_mu} with $\mu_i, \mu_i'\not\in \{1,-(k-1)\}$. Then, for $d> k \geq 3$, with high probability, for all $i\in [n]$, 
     \begin{align}
        {\max}\{ \|w_i\|_{\infty}, \|w_i'\|_{\infty} \} \leq  \frac{  \sqrt{k-1} +o(1)}{  \sqrt{d-1 } - \sqrt{k-1}  } \|v_i\|_{\infty}.
     \end{align}
 \end{enumerate}
\end{theorem}

Theorem~\ref{thm:bi_tildeB} shows that eigenvector delocalization of $A$ implies eigenvector delocalization for $B$ and $\tilde B$. However, no eigenvector delocalization results are available in the literature for the adjacency matrix random regular hypergraphs. It's possible that by adapting the result of eigenvector delocalization for bipartite biregular graphs \cite[Corollary 2.8]{yang2017bulk}  together with the connection between random regular hypergraphs and random bipartite biregular graphs established in \cite{dumitriu2021spectra}, one can obtain an eigenvector delocalization bound for the adjacency matrix in the random regular hypergraph case.

\subsection{Community detection in the RSBM}
\begin{definition}[Regular stochastic block model (RSBM)]\label{def:RSBM}
  For an even integer $n$  and two integers $d_1$ and $d_2$, the \textit{regular stochastic block model} with
vertex set $[n]$ is obtained as follows. Choose a partition  $V_1,V_2$ of equal size $n/2$ of the
vertex set $[n]$, uniformly from among the set of all such partitions. Choose two independent copies of
uniform  $d_1$-regular graphs with vertex set $V_1$, respectively $V_2$. Finally, connect the vertices
from $V_1$ with those from $V_2$ by a uniformly random $d_2$-bipartite regular graph.  
\end{definition}

Let $\sigma\in \{-1,1\}^{n}$ be a vector for the community assignment defined by 
\begin{align}
    \sigma_i=\begin{cases}
        1 & i\in V_1\\
        -1 & i\in V_2.
    \end{cases}
\end{align}
The \textit{community detection} problem for the regular stochastic block model is to observe a random graph sampled from the model and construct an estimator for the community assignment vector $\sigma$.

Note that a graph sampled from this model is a random  $(d_1+d_2)$-regular graph but not uniformly distributed. {It has been shown in} \cite[Proposition 1]{brito2015recovery} that the RSBM is asymptotically distinguishable from a uniformly chosen random regular graph with the same degree.  However, some of our results for deterministic regular graphs can still be applied to this model.  We describe the real eigenvalues and eigenvectors of $\tilde B$ as follows:

\begin{theorem}[Real eigenvalues and eigenvectors of $\tilde B$]\label{thm:RSBM}
    Let $\tilde B$ be the reduced non-backtracking operator of an $(n,d_1,d_2)$-regular stochastic block model. Then,  the following holds:
    \begin{enumerate}
        \item $d_1+d_2-1$ is an eigenvalue of $\tilde B$ with the corresponding eigenvector $v_1=\left[ 1,\dots,1\right]^\top\in \mathbb R^{2n}$ and $1$ is an eigenvalue of $\tilde B$ with the corresponding eigenvector \[v_1'=\left[1,\dots, 1, \frac{1}{d_1+d_2-1},\dots, \frac{1}{d_1+d_2-1}\right]^\top\in \mathbb R^{2n}.\]
        \item When $(d_1-d_2)^2>4(d_1+d_2-1)$,  there are two  real eigenvalues of $\tilde B$ given by 
        \begin{align}\label{eq:mu2formula}
            \mu_2,\mu_2'=\frac{d_1-d_2\pm\sqrt{(d_1-d_2)^2-4(d_1+d_2-1)}}{2},
        \end{align}
        with the corresponding eigenvector
        \begin{align}
            v_2=\begin{bmatrix}
\sigma\\
 \frac{\mu_2}{d_1+d_2-1} \sigma 
\end{bmatrix} , \quad v_2'=\begin{bmatrix}
\sigma\\
 \frac{\mu_2'}{d_1+d_2-1} \sigma
\end{bmatrix} .
        \end{align}
    \item When $(d_1-d_2)^2>4(d_1+d_2-1)$ and $d_1$ is even, with high probability, {the four values} $d_1+d_2-1, 1, \mu_2$ and $\mu_2'$ are real eigenvalues {of $\tilde B$} with multiplicity one, and the rest of the eigenvalues of $\tilde B$ are within $o(1)$ distance from the circle of radius $\sqrt{d_1+d_2-1}$.
    \end{enumerate}
\end{theorem}

\begin{figure} 
    \centering
    \includegraphics[width=\linewidth]{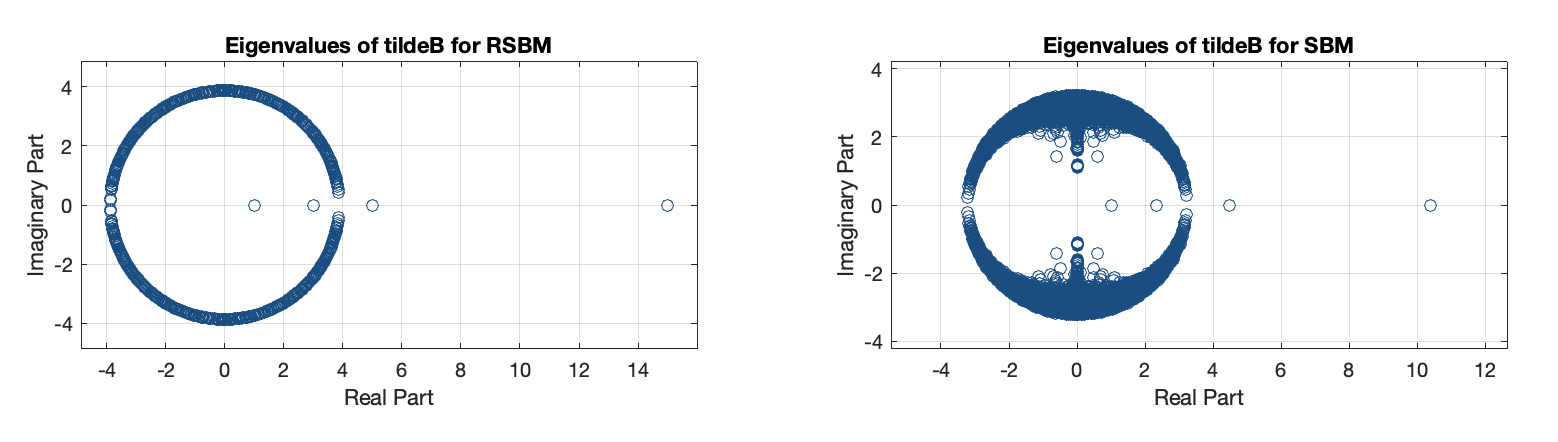}
    \vspace{-20pt}
    \caption{On the left is the non-backtracking spectrum of an RSBM, where we choose $d_1 = 12$ and $d_2 = 4$ with size 4000. There is an eigenvalue ${\mu_2'}$ inside the circle matched with \eqref{eq:mu2formula} given by $\frac{\mu_1}{\mu_2}$.   On the right is the non-backtracking spectrum of a stochastic block model where we choose $a = 15$ and $b=6$ with size 4000. The insider real eigenvalue is close to $\frac{\mu_1}{\mu_2}$.} 
    \label{fig:conjectureSBM} 
\end{figure}

See Figure~\ref{fig:conjectureSBM} for {a} simulation on the spectrum of an RSBM and an Erd\H{o}s-R\'{e}nyi SBM.
From Theorem~\ref{thm:RSBM}, we obtain the following corollary for exact recovery with a spectral method in the RSBM.
\begin{cor}[Exact recovery in the RSBM] \label{cor:recovery} Assume $(d_1-d_2)^2>4(d_1+d_2-1)$ and $d_1$ is even.
Let $u_2=\begin{bmatrix}
    x_2\\
    y_2
\end{bmatrix}, u_2'=\begin{bmatrix}
    x_2'\\
    y_2'
\end{bmatrix} $ be {eigenvectors} of $\tilde B$ corresponding to eigenvalue $\mu_2 $ and $\mu_2'$, {respectively}. Then, with high probability, each of the four vectors $x_2, y_2,x_2', y_2'$ is a scalar multiple of the community assignment {vector} $\sigma$, namely, $x_2, y_2,x_2', y_2' \in \mathrm{span}\{\sigma\} $.
\end{cor}
Corollary~\ref{cor:recovery} implies that both eigenvectors  $u_2$ and $u_2'$ of $\tilde B$ can exactly recover the community assignment vector $\sigma$ with high probability. Such a spectral method based on the eigenvector associated with the insider eigenvalue for the Erd\H{o}s-R\'{e}nyi SBM was conjectured in \cite{dall2019revisiting} and remains open. Compared with the regular stochastic block model, the key difference for the Erd\H{o}s-R\'{e}nyi stochastic block model is the lack of degree concentration in the very sparse regime, which makes it challenging to establish a direct connection between eigenvectors of $\tilde B$ and $A$.

\begin{remark}[The parity constraint of $d_1$]
  In Theorem~\ref{thm:RSBM}, the constraint that $d_1$ is even is a technical condition due to the proof technique in \cite{brito2015recovery}, which involves random lifts studied in \cite{bordenave2020new,friedman2014relativized}. It was conjectured in \cite{brito2015recovery} that such a restriction can be removed. We expect by adapting the  proof techniques in \cite{bordenave2020new,brito2021spectral}, Theorem~\ref{thm:RSBM} can {be} shown for all $(d_1-d_2)^2>4(d_1+d_2-1)$.  We leave it as a future direction.  
\end{remark}

\section{Proof of Theorem \ref{thm:real}}\label{sec:thm1}
{We first derive an algebraic relation between eigenvalues and eigenvectors of $A$ and $B$ when $G$ is a $d$-regular graph, from which the known spectral laws for $A$ give the limiting distribution of the real parts of the eigenvalues of $\tilde B$.}

\begin{lemma}[Spectral relation between $A$ and $\tilde B$]\label{lem:eiganvalues}
Let $G$ be a $d$-regular graph with adjacency matrix $A$ and non-backtracking matrix $B$ and $d\geq 2$.
Let $v_i$ be an eigenvector of $A$ corresponding to {eigenvalue} $\lambda_i $ for $ i\in [n]$.  
Then
\begin{equation}
    \mu_i, \mu_i' = \frac{\lambda_i \pm \sqrt{{\lambda_i}^2-4(d-1)}}{2}
\end{equation}
are two eigenvalues of $\tilde B$, with two corresponding eigenvectors of the form  
\begin{equation}\label{eq:eigenvectortb}
u_i = 
    \begin{bmatrix}
v_i\\
 \frac{\mu_i}{d-1} v_i
\end{bmatrix}  ,\quad u_i' = 
    \begin{bmatrix}
v_i\\
 \frac{\mu_i'}{d-1} v_i
\end{bmatrix} .
\end{equation}
\end{lemma}

\begin{proof}
Let $\begin{bmatrix}
x\\
y
\end{bmatrix} $ {for} $ x,y\in \mathbb C^{n}$ be an eigenvector of $\tilde B$ with respect to an eigenvalue $\mu$. {Then} we have  $$ \mu \begin{bmatrix}
x\\
y
\end{bmatrix} = \begin{bmatrix}
0 &  (d-1) I\\
-I  & A
\end{bmatrix} \begin{bmatrix}
x\\
y
\end{bmatrix}. $$
This equation implies \begin{align}\label{eq:eigen}
    y =  \frac{\mu}{d-1} x , \quad  Ax = \left(\mu + \frac{d-1}{\mu}\right)x .
\end{align}

For each eigenvalue $\lambda_i$ of $A$ with a corresponding eigenvector $v_i\in \mathbb R^n$, there are two corresponding {eigenvalues} $ \mu_i, \mu'_i$ {of $\tilde B$} satisfying the quadratic equation \begin{equation}\label{eq:mu}
    x^2 - \lambda_i x + d -1 = 0.
\end{equation}
In particular, $\lambda_1=d$ gives two real eigenvalues $\mu_1=d-1$ and $\mu_1'=1$.

By counting multiplicity, Equation \eqref{eq:mu} gives all the $2n$ eigenvalues of $\tilde B$. Thus, by Equations~\eqref{eq:eigen} and ~\eqref{eq:mu}, there are $n$ eigenvectors of $\tilde B$ of the form  \begin{equation}\label{eigenvector1}
u_i=    \begin{bmatrix}
v_i\\
 \frac{\mu_i}{d-1} v_i
\end{bmatrix} 
\text{, where } \mu_i = \frac{\lambda_i+\sqrt{{\lambda_i}^2-4(d-1)}}{2},
\end{equation}
and other $n$ eigenvectors of the form,
\begin{equation}\label{eigenvector2}
 u_i'=\begin{bmatrix}
v_i\\
 \frac{\mu_i'}{d-1} v_i
\end{bmatrix} 
\text{, where } \mu_i' = \frac{\lambda_i-\sqrt{{\lambda_i}^2-4(d-1)}}{2}.
\end{equation}
\end{proof}

With Lemma \ref{lem:eiganvalues},  we are ready to prove Theorem \ref{thm:real} for random $d$-regular graphs.
\begin{proof}[Proof of Theorem \ref{thm:real}]
From Equation \eqref{eq:mu}, when $|\lambda_i|^2\leq 4(d-1)$, we have 
\[\mu_i+ \mu_i'=\lambda_i, \quad |\mu_i|=|\mu_i'|=\sqrt{d-1},\] and   the real parts of $\mu_i,\mu_i'$ satisfy
 \begin{align}\label{eq:quadratic_coefficient}
 x_i = x_i' = \frac{\lambda_i}{2}.
 \end{align}
 
The Kesten-McKay law  in \cite{MCKAY1981203} shows that with high probability,  the limiting eigenvalue distribution of the adjacency matrix for a random $d$-regular graph converges weakly to a probability measure with density
\[ f(x)=\frac{d\sqrt{4(d-1)-x^2}}{2\pi (d^2-x^2)} \mathbf{1} \left\{ |x|\leq 2\sqrt{d-1}\right\}.\]

From the Kesten-McKay law above, with high probability, there are $n-o(n)$ eigenvalues $\lambda_i$ {that} satisfy $|\lambda_i|^2\leq 4(d-1)$. Therefore, with high probability, there are $2n-o(n)$ eigenvalues of $\tilde B$ inside the circle of radius $\sqrt{d-1}$, and $n$ pairs of $x_i,x_i'$ satisfy \eqref{eq:quadratic_coefficient}. Therefore, the empirical measure of the real parts of $\mu_i$ converges weakly in probability to a rescaled Kesten-McKay law given by 
    \begin{equation}   
   \mu_{\mathrm{KM}}(x)=   \frac{ 2d\sqrt{(d-1)-x^2} }{\pi(d^2 - 4x^2)} \mathbf{1} \left\{ |x|\leq \sqrt{d-1}\right\}.
    \end{equation}

{Recall that $A$ is the adjacency matrix for a random $d$-regular graph on $n$ vertices;} when $d\to\infty$ as $n\to\infty$, it is shown in \cite{tran2013sparse,dumitriu2012sparse} that the empirical spectral distribution of $\frac{A}{\sqrt{d-1}}$ converges weakly in probability to a semicircle law. With  Equation \eqref{eq:quadratic_coefficient}, following  the same argument  as in the case for fixed $d$, the empirical measure  $\frac{1}{2n}\sum_{i=1}^{2n} \delta_{2x_i/\sqrt{d-1}}$ converges weakly in probability to the semicircle law.
\end{proof}

\section{Proof of Theorem \ref{thm:tildeB}}

{Building on the algebraic correspondence established in Lemma~\ref{lem:eiganvalues}, we obtain the eigenvector delocalization bound for $A$ from \cite{huang2021spectrum} to $\tilde B$.}

\begin{lemma}[Theorem 1.4 in \cite{huang2021spectrum}] \label{lem:A_delocalization}
{Let $A$ be the adjacency matrix of a $d$-regular graph.} Then, for any fixed $d\geq 3$, there exists absolute constants $C_1, C_2>0$ such that with probability at least $1-n^{-C_1}$, all unit eigenvectors $v_i$ of $A$ {satisfy}
    \begin{align}
        \|v_i\|_{\infty}\leq \frac{\log^{C_2}(n)}{\sqrt n}.
    \end{align}
\end{lemma}


\begin{proof}[Proof of Theorem~\ref{thm:tildeB}]
    From \eqref{eq:eigenvectortb}, any eigenvector $u_i$ of $\tilde B$ associated with $v_i$ of $A$ satisfies
    \begin{align}
        \frac{\|u_i\|_{\infty}}{\|u_i\|_2} \leq \frac{\max \left\{ 1, \frac{|\mu_i|}{d-1}\right\}}{\sqrt{1+\frac{|\mu_i|^2}{(d-1)^2}}}\frac{\|v_i\|_{\infty}}{\|v_i\|_2}\leq \frac{\|v_i\|_{\infty}}{\|v_i\|_2}.
    \end{align}
    Then Theorem~\ref{thm:tildeB} follows from Lemma~\ref{lem:A_delocalization}.
\end{proof}

\section{Proof of Theorem \ref{thm:B}}
\label{sec:thm3}

{In this section, we derive a deterministic correspondence and an $\ell_\infty$-norm comparison between the eigenvectors of $A$ and $B$, and then use spectral gap and adjacency eigenvector delocalization estimates to prove complete delocalization of the nontrivial eigenvectors of $B$.}
We first introduce a lemma contained in \cite{lubetzky2016cutoff} for a deterministic relation between eigenvectors of $B$ and eigenvectors of $A$ and provide a proof for completeness.

\begin{lemma}[Remark 3.4 in \cite{lubetzky2016cutoff}] \label{lem:directededge}
{Let}  $\vec{E}$ be the oriented edge set for a $d$-regular graph $G=(V,E)$ defined in \eqref{eq:def_oriented}. Let $\mu_i, \mu_i'$ denote each eigenvalue pair of $\tilde B$ corresponding to eigenvalue $\lambda_i$ of $A$ through the quadratic equation \eqref{eq:mu} with $|\lambda_i|\not=d$. Let $v_i$ be a unit eigenvector of $A$ associated $\lambda_i$. 
Then, each of nonzero vectors $\tilde w_i, \tilde{w}_i' \in \mathbb C^{nd} $ is an eigenvector of $B$ associated with $\mu_i, \mu_i'$, respectively, {if and only if}  \begin{equation}\label{eq:eigenvectorB}
    \tilde w_i(x,y) = \mu_i v_i(y) - v_i(x), \quad \tilde w_i'(x,y) = \mu_i' v_i(y) - v_i(x), \quad\forall (x,y)\in \vec{E}.
\end{equation} 
\end{lemma}

\begin{proof}
   It suffices to check $\tilde w_i, i\in [n]$, and the same argument works for $\tilde w_i'$. Let $\tilde w_i$ be defined as \eqref{eq:eigenvectorB}.  For any $(x,y)\in \vec{E}$, we are able to calculate 
 \begin{align*}
    (B\tilde w_i)(x,y) &= \sum_{\substack{z:(y,z) \in \vec{E} \\ z \neq x}} (\mu_i v_i(z) - v_i(y)) \\
    &= \mu_i [(Av_i)(y) - v_i(x)] - (d-1) v_i(y) \\
    &= [\mu_i \lambda_i - (d-1)] v_i(y) - \mu_i v_i(x) \\
    &= [\mu_i^2  + d -1 - (d-1)]v_i(y) - \mu_i v_i(x) \\
    &= \mu_i \tilde w_i(x,y),
\end{align*}
where in the fourth line $ \mu_i \lambda_i$ is replaced by $ \mu_i^2  + d -1 $  due to  the quadratic equation \eqref{eq:mu}. Then, we have  $B\tilde w_i = \mu_i \tilde w_i$. When $|\lambda_i|\not= d$, {we have}  $\mu_i\not=\pm 1$. {We can check $\tilde{w}_i\not=0$. In particular, by contradiction, suppose $\tilde{w}_i =0$, then \begin{align}
    \tilde w_i(x,y) = \mu_i v_i(y) - v_i(x) = 0 \quad\Rightarrow\quad \mu_i v_i(y) = v_i(x), \quad\forall (x,y) \in \vec{E}.
\end{align}
Hence, since $(x,y), (y,x) \in \vec{E}$, we have $v_i(x) = \mu_i v_i(y)$ and $v_i(y) = \mu_i v_i(x) $, which together gives \begin{align}
   v_i(x) = \mu_i (\mu_i v_i(x)) = \mu_i^2 v_i(x).
\end{align}
Contradict to the fact that $\mu_i \neq \pm 1$.} Therefore, we have $w_i$ is an eigenvector of $B$ with the corresponding eigenvalue $ \mu_i$.
{Conversely, let $w$ be an eigenvector of $B$ associated with $\mu_i  \neq \pm 1$ for some $i \in [n]$. Then, for any $(x,y) \in \vec{E} $, we have \begin{align}
    \mu_i w(x,y) = B w(x,y) 
    &=  \sum_{\substack{z:(y,z) \in \vec{E} \\ z \neq x}}  w(y,z) \\
    &= \sum_{z:(y,z) \in \vec{E} }  w(y,z) - w(y,x) 
    \intertext{Similar, we can write $w(y,x)  = (1/\mu_i) \sum_{z:(x,z) \in \vec{E} } w(x,z) - (1/\mu_i) w(x,y)$, then}
    &= \sum_{z:(y,z) \in \vec{E} }  w(y,z) - \frac{1}{\mu_i} \sum_{z:(x,z) \in \vec{E} } w(x,z) + \frac{1}{\mu_i} w(x,y).
\end{align}
By rearranging, \begin{align}
     w(x,y) &= \frac{\mu_i}{\mu_i^2-1} \sum_{z:(y,z) \in \vec{E} }  w(y,z) - \frac{1}{\mu_i^2-1} \sum_{z:(x,z) \in \vec{E} } w(x,z) 
     \intertext{Since $\mu_i \neq \pm 1$, we denote $v_i(x) = 1/ (\mu_i^2-1 ) \sum_{z:(x,z) \in \vec{E} }  w(x,z) $, then}
     &= \mu_i v_i(y) - v_i(x).
\end{align}
To verify $v_i$ is eigenvector of $A$ associated to $\lambda_i$, we note \begin{align}
    (\mu_i^2-1 ) v_i(x)  = \sum_{z:(x,z) \in \vec{E} }  w(x,z) =  \sum_{z:(x,z) \in \vec{E} }( \mu_i v_i(z) - v_i(x)) = \mu_i (Av_i)(x) - d_i(x),
\end{align}
and hence, by \eqref{eq:mu}, \begin{align}
     Av_i= \left(\mu_i+\frac{d-1}{\mu_i}\right)v_i = \lambda_i v_i.
\end{align}}
\end{proof}
We can derive a deterministic $\ell_{\infty}$-norm bound for eigenvectors of $B$ constructed from Lemma \ref{lem:directededge} as follows.
\begin{lemma}[Deterministic $\ell_{\infty}$-norm bound for eigenvectors of $B$] \label{lem:deterministic_graph}
Let $B$ be the non-backtracking operator of a connected $d$-regular graph, and let $ \tilde w_i$ be eigenvectors of $B$ 
associated with $ \mu_i$ constructed in Lemma \ref{lem:directededge} where $\mu_i\notin \{1,d-1\}$.   Then 
\begin{align}\label{eq:deterministic_delocalization}
    \frac{\|\tilde w_i\|_{\infty}}{\|\tilde w_i\|_2} \leq \frac{\|v_i\|_{\infty} (|\mu_i|+1)}{\sqrt{d^2-\lambda_i^2}}.
\end{align}
The same bound holds for $\tilde w_i'$.
\end{lemma}
\begin{proof}
It suffices to prove the estimate for $\tilde w_i$.    From \eqref{eq:eigenvectorB},
\begin{equation}\label{eq:infinite_norm}
    \| \tilde w_i \|_{\infty} \leq \sup_{(x,y)\in \overrightarrow{E}} |\mu_i v_i(y) - v_i(x)| 
    \leq \| v_i \|_{\infty} ( |\mu_i|  +1 ).
\end{equation}
     On the other hand, since $\mu_i,\mu_i'$ is a conjugate pair from \eqref{eq:mu},
     \begin{align}
    \| \tilde w_i \|_2^2 
    &= \sum_{(x,y)\in \overrightarrow{E}} ( \mu_i v_i(y) - v_i(x))(\mu_i' v_i(y)-v_i(x)) \\
    &= |\mu_i|^2 \sum_{(x,y)\in \overrightarrow{E}} (v_i(y))^2 + \sum_{(x,y)\in \overrightarrow{E}} (v_i(x))^2 - (\mu_i+\mu_i') \sum_{(x,y)\in \overrightarrow{E}} v_i(y)v_i(x)\\
    &= d|\mu_i|^2 \|v_i\|_2^2+ d\|v_i\|_2^2 -(\mu_i+\mu_i') \sum_{(x,y)\in \overrightarrow{E}} v_i(y)v_i(x)\\
    &=d^2-\lambda_i \sum_{(x,y)\in \overrightarrow{E}} v_i(y)v_i(x).
\end{align}
Moreover, 
\begin{align}
    \sum_{(x,y)\in \overrightarrow{E}} v_i(y)v_i(x)=\sum_{x\in V} (Av_i)(x) v_i(x)=\lambda_i\sum_{x\in V}  v_i^2(x)=\lambda_i.
\end{align}
This gives us 
\begin{align}\label{eq:lowerbound_2}
    \| \tilde w_i \|_2^2=d^2-\lambda_i^2.
\end{align}
Since $\mu_i\notin\{1,d-1\}$, we have $\lambda_i\not=d$.
Equations \eqref{eq:infinite_norm} and \eqref{eq:lowerbound_2} imply \eqref{eq:deterministic_delocalization}.
\end{proof}

Now, we are ready to prove Theorem \ref{thm:B}. 

\begin{proof}[Proof of Theorem~\ref{thm:B}]

 \textit{Case 1: $\mu_i=d-1$.} The deterministic eigenvalue $\lambda_1=d$ of $A$ with eigenvector $v_1=[1,\dots, 1]^\top \in \mathbb R^n$ yields two eigenvalues $\mu_1=d-1, \mu_1'=1$. Let  $\tilde w_1$ be the eigenvector given by \eqref{eq:eigenvectorB} associated with $\mu_1$. One can directly check that 
      \begin{align}
          \frac{ \|\tilde w_1\|_{\infty}}{\|\tilde w_1\|_2}= \frac{1}{\sqrt{nd}}.
      \end{align}

\textit{Case 2: $\mu_i\notin\{\pm 1,d-1\}$.} By the spectral gap bound from \cite[Theorem 1]{bordenave2020new},  we have for some constants $C,C_1>0$, with probability at least $1-n^{-C}$, for any eigenvalue $\mu$ of $\tilde B$ with $\mu\not=d-1$,
\begin{align}\label{eq:spectralgap}
    |\mu|\leq \sqrt{d-1}+C_1\left( \frac{\log \log n}{\log n}\right)^2,
\end{align} 
and for any eigenvalue $\lambda$ of $A$ with $\lambda \neq d$, \begin{align}\label{eq:spectralgap_lambda}
    |\lambda| \leq 2\sqrt{d-1}+C_1\left( \frac{\log \log n}{\log n}\right)^2.
\end{align}
Therefore with \eqref{eq:deterministic_delocalization} , with probability at least $1-n^{-C}$, \begin{equation}
     \frac{\|\tilde w_i\|_{\infty}}{\|\tilde w_i\|_2} \leq \frac{C'(\sqrt{d-1}+1) }{ d-2 } \|v_i\|_{\infty} \leq \frac{2C'}{ \sqrt{d}}   \|v_i\|_{\infty},
\end{equation}
      for an absolute constant $C'>0$. 
 The delocalization bound of $v_i$ in Lemma \ref{lem:A_delocalization} implies that with probability at least $1-n^{-C_1}$, 
       \begin{align} \label{eq:ratiobound}
       \frac{\|\tilde w_i\|_{\infty}}{\|\tilde w_i\|_2} \leq {\frac{\log^{C_2}(n)}{\sqrt{nd}}}.
       \end{align}
       With the two cases discussed above, taking $\ell_2$-normalized eigenvector $w_i=\frac{\tilde w_i}{\|\tilde w_i\|_2}$ for $i\in [n]$ completes the proof. 
\end{proof}

\section{Proof of Theorem \ref{thm:bi_real}}

{Similar to Section~\ref{sec:thm1}, we first establish a spectral relation between $A$ and $\tilde B$ for a regular hypergraph and apply the limiting spectral distribution of $A$ to $\tilde B$ in the three regimes of $(d,k)$.}

\begin{lemma}[Spectral relation between $A$ and $\tilde B$]\label{lem:bi_eigenvalues}
Let $H = (V, E)$ be a $k$-uniform, $d$-regular hypergraph with adjacency matrix $A$ and reduced  non-backtracking matrix $\tilde B$. Let $\lambda_i$ be an eigenvalue of $A$. Then each $\lambda_i$ corresponds to two eigenvalues $\mu_{i},\mu_i'$ of $\tilde B$, which satisfy the equation
\begin{equation}\label{eq:bi_mu}
    \mu^2 - (\lambda_i-k+2 )\mu + (d-1)(k-1) = 0,
\end{equation}
and {which} have corresponding eigenvectors 
\begin{equation}\label{eq:bi_eigenvector_H}
u_i = 
    \begin{bmatrix}
v_i\\
 \frac{\mu_i}{d-1} v_i
\end{bmatrix}  ,\quad u_i' = 
    \begin{bmatrix}
v_i\\
 \frac{\mu_i'}{d-1} v_i
\end{bmatrix} .
\end{equation}
\end{lemma}

\begin{proof}
{Let} $\begin{bmatrix}
x\\
y
\end{bmatrix} $ for $ x,y\in \mathbb C^{n}$ be an eigenvector of $\tilde B$ with respect to an eigenvalue $\mu$, we have  $$ \mu \begin{bmatrix}
x\\
y
\end{bmatrix} = \begin{bmatrix}
0 &  (d-1) I\\
-(k-1)I  & A-(k-2)I
\end{bmatrix} \begin{bmatrix}
x\\
y
\end{bmatrix}. $$
This equation implies \begin{align}
    Ax = \left(\mu + \frac{(k-1)(d-1)}{\mu}+k-2\right)x, \quad   y =  \frac{\mu}{d-1} x.
\end{align}
For each eigenvalue $\lambda_i$ of $A$ with a corresponding eigenvector $v_i\in \mathbb R^n$, there are two corresponding {eigenvalues} $ \mu_i, \mu'_i$ {of $\tilde B$} satisfying the quadratic equation \begin{equation}
    \mu^2 - (\lambda_i-k+2) \mu + (d -1)(k-1) = 0.
\end{equation}
This gives all the $2n$ eigenvalues of $\tilde B$. The rest of the proof follows in the same way as in Lemma~\ref{lem:eiganvalues}.
\end{proof}

\begin{proof}[Proof of Theorem~\ref{thm:bi_real}]
    From Lemma~\ref{lem:bi_eigenvalues}, all eigenvalues of $\tilde B$ {satisfy} \eqref{eq:bi_mu}.
When \[|\lambda_i-k+2|^2\leq 4(d-1)(k-1),\] we have 
\[\mu_i+ \mu_i'=\lambda_i-k+2, \quad |\mu_i|=|\mu_i'|=\sqrt{(d-1)(k-1)},\] and   the real parts of $\mu_i,\mu_i'$ satisfy 
 \begin{align}\label{eq:bi_quadratic_coefficient}
 x_i = x_i' = \frac{\lambda_i {-k+2}}{2}.
 \end{align}

 For different  regimes of $(d,k)$, the limiting spectral distribution of the normalized adjacency matrix $\tilde A:=\frac{A-(k-2)}{\sqrt{(d-1)(k-1)}}$ for random regular hypergraphs were obtained in \cite[Corollary 6.11]{dumitriu2021spectra}:
 \begin{itemize}
     \item If $d,k$ are fixed,   the empirical spectral distribution of $\tilde{A}$ converges in probability to a probability measure supported on $[-2,2]$ whose density function is given by \begin{align}
	f(x)=\frac{1+\frac{k-1}{q}}{(1+\frac{1}{q}-\frac{x}{\sqrt{q}})(1+\frac{(k-1)^2}{q}+\frac{(k-1)x}{\sqrt{q}})}\frac{1}{\pi}\sqrt{1-\frac{x^2}{4}},
	\end{align}
	where $q=(k-1)(d-1)$.
	 \item For $d,k\to\infty$ with $\frac{d}{k}\to\alpha\geq 1$ and $d\leq \frac{n}{32}$,  the empirical spectral distribution of $\tilde{A}$ converges in probability to  a measure supported on $[-2,2]$ with a density function given by	
	 \begin{align}\label{mua}
	g(x)=\frac{\alpha}{1+\alpha+\sqrt{\alpha}x}\frac{1}{\pi}\sqrt{1-\frac{x^2}{4}} .
	\end{align}
    \item If $d\to\infty, d=o(n^{\epsilon})$ for any $\epsilon>0$ and $\frac{d}{k}\to\infty$,  the empirical spectral distribution of $\tilde A$ converges to the semicircle law in probability.
 \end{itemize}
 With the relation \eqref{eq:bi_quadratic_coefficient}, we can follow the same argument as in the proof of Theorem~\ref{thm:real} to complete the proof of Theorem~\ref{thm:bi_real}.
\end{proof}

\section{Proof of Theorem~\ref{thm:bi_tildeB}} 

{In this section, we derive the $\ell_\infty$-norm bounds for $\widetilde{B}$ and $B$ separately using their eigenvector relations with $A$. By Lemma~\ref{lem:bi_eigenvalues}, the bound for $\widetilde{B}$ follows directly from \eqref{eq:bi_eigenvector_H}.}

\begin{proof}[Proof of Theorem~\ref{thm:bi_tildeB}~(1)]
      From \eqref{eq:bi_eigenvector_H},  the corresponding unit eigenvectors $u_i, u_i'$ {satisfy} 
\[ \|u_i\|_{\infty}, \|u_i'\|_{\infty}= \frac{\max \left\{ 1, \frac{|\mu_i|}{d-1}\right\}}{\sqrt{1+ \frac{|\mu_i|^2}{(d-1)^2}}} \|v_i\|_{\infty}\leq \|v_i\|_{\infty}. \]
This proves the first statement of Theorem~\ref{thm:bi_tildeB}.
\end{proof}
  
We now consider the second statement of Theorem~\ref{thm:bi_tildeB}. The following lemma provides a relation between eigenvectors of $B$ and $A$ for regular hypergraphs.

\begin{lemma}[Eigenvectors of $B$ for regular hypergraphs]\label{lem:eigenvector_hypergraph}
{Let}  $\vec{E}$ be the oriented hyperedge set for a $(d,k)$-regular hypergraph $H=(V,E)$. Let $\mu_i, \mu_i'$ denote each eigenvalue pair of $\tilde B$ corresponding to eigenvalue $\lambda_i$ of the adjacency matrix $A$, through the quadratic equation \eqref{eq:bi_mu}. Let $v_i$ be a unit eigenvector of $A$ associated $\lambda_i$. Define $\tilde{w}_i, \tilde{w}_i'\in \mathbb C^{nd}$ such that for any $(x,e)\in \vec{E}$,
\begin{align} \label{eq:w_H}
    \tilde w_i(x,e) :&=\mu_i \left(\sum_{y\in e, y\not=x} v_i(y)\right)-(k-1) v_i(x),\\  \tilde w_i'(x,e) :&=\mu_i' \left(\sum_{y\in e, y\not=x} v_i(y)\right)- (k-1)v_i(x).
\end{align}
{If} $\tilde w_i, \tilde{w}_i'\not=0$, {then  each of} $\tilde w_i, \tilde{w}_i'$ is an eigenvector of $B$ associated with $\mu_i,\mu_i'$, respectively. 
\end{lemma}

\begin{proof}
   It suffices to consider $\tilde w_i, i\in [n]$. With \eqref{eq:w_H} and the definition of $B$ for hypergraphs, we have for any $(x,e)\in \vec{E}$,
 \begin{align}
    &(B\tilde w_i)(x,e) = \sum_{(y,f)\in \vec{E}:y\in e, y\not=x, f\not=e} \tilde{w}_i(y,f)\\
    &=\sum_{(y,f)\in \vec{E}:y\in e, y\not=x, f\not=e} \left( \mu_i \left(\sum_{z\in f, z\not=y} v_i(z)\right)-(k-1)v_i(y)\right)\\
    &=\mu_i \left(\sum_{(y,f)\in \vec{E}:y\in e, y\not=x, f\not=e} \sum_{z\in f,z\not=y} v_i(z)\right) -(d-1)(k-1) \sum_{y\in e, y\not=x} v_i(y)\\
    &=\mu_i \left( \sum_{y\in e, y\not=x} (Av_i)(y)-\sum_{y\in e, y\not=x}\sum_{z\in e,z\not=y} v_i(z) \right)-(d-1)(k-1) \sum_{y\in e, y\not=x} v_i(y)\\
    &=\mu_i \left( \sum_{y\in e, y\not=x} \lambda_i v_i(y)-\left(\sum_{y\in e, y\not=x}(k-2)v_i(y) \right)-(k-1)v_i(x) \right)\\
    &\quad -(d-1)(k-1) \sum_{y\in e, y\not=x} v_i(y)\\
    &=(\mu_i\lambda_i-(k-2)\mu_i-(d-1)(k-1))\left(\sum_{y\in e, y\not=x} v_i(y)\right)-\mu_i (k-1) v_i(x) \\
    &=\mu_i^2\left(\sum_{y\in e, y\not=x} v_i(y)\right)-\mu_i (k-1) v_i(x) \label{eq:line3}\\
    &=\mu_i \left(\mu_i \sum_{y\in e, y\not=x} v_i(y)-(k-1)v_i(x)\right) =\mu_i \tilde{w}_i(x,e),
\end{align}
where in \eqref{eq:line3}, we use the relation between $\mu_i,\lambda_i$ in \eqref{eq:bi_mu}. This implies  $\tilde{w}_i$ is an eigenvector of $B$ associated with $\mu_i$.
\end{proof}

Lemma~\ref{lem:eigenvector_hypergraph} can be used to prove the following deterministic $\ell_{\infty}$-norm bound for eigenvectors. When $k=2$, Lemma~\ref{lem:deterministic_H} reduces to Lemma~\ref{lem:deterministic_graph}.

\begin{lemma}[Deterministic $\ell_{\infty}$-norm bound for eigenvectors of $B$]
\label{lem:deterministic_H}
    Let $B$ be the non-backtracking operator of a $(d,k)$-regular hypergraph. Let $ \tilde w_i$ be the eigenvector of $B$ 
associated with $ \mu_i$ constructed in Lemma \ref{lem:eigenvector_hypergraph} with $\mu_i\not\in \{ 1, (d-1)(k-1)\}$ and $\lambda_i$ be the corresponding eigenvalue of $A$.   Then 
\begin{align}\label{eq:deterministic_delocalization_2}
    \frac{\|\tilde w_i\|_{\infty}}{\|\tilde w_i\|_2} \leq  \frac{ \sqrt{k-1}(|\mu_i|+1) }{\sqrt{(d+\lambda_i)(d(k-1)-\lambda_i)} } \| v_i \|_{\infty}.
\end{align}
The same bound holds for $\tilde w_i'$.
\end{lemma}
\begin{proof}
    It suffices to prove the estimate for $\tilde w_i$.    From \eqref{eq:w_H},
\begin{equation}\label{eq:infinite_norm_H}
    \| \tilde w_i \|_{\infty}  \leq \sup_{(x,e)  \in \vec{E}} \left| \mu_i \bigg(\sum_{y\in e, y\not=x} v_i(y) \bigg)-(k-1) v_i(x) \right|
    \leq \| v_i \|_{\infty} \bigg( |\mu_i|(k-1) +k-1 \bigg).
\end{equation}
     On the other hand, since $\mu_i,\mu_i'$ is a conjugate pair in \eqref{eq:bi_mu},
     \begin{align}
    \| \tilde w_i \|_2^2 
    &= \sum_{(x,e)  \in \vec{E}} \left( \mu_i \bigg(\sum_{y\in e, y\not=x} v_i(y) \bigg)-(k-1) v_i(x) \right) \\
    &\quad \cdot\left( \mu_i' \bigg(\sum_{y\in e, y\not=x} v_i(y)\bigg)-(k-1) v_i(x) \right) \\
     &= |\mu_i|^2 \sum_{(x,e)  \in \vec{E}} \bigg(\sum_{y\in e, y\not=x} v_i(y) \bigg)^2 + (k-1)^2 \sum_{(x,e)  \in \vec{E}} v_i(x)^2 \\
    &\quad - (\mu_i+\mu_i')(k-1) \sum_{(x,e)  \in \vec{E}} \bigg(\sum_{y\in e, y\not=x} v_i(y) \bigg) v_i(x)\\
    &=(d-1)(k-1)\sum_{(x,e)  \in \vec{E}} \bigg(\sum_{y\in e, y\not=x} v_i(y) \bigg)^2 +(k-1)^2 \sum_{(x,e)  \in \vec{E}} v_i(x)^2\\
    &\quad -(\lambda_i-k+2)\sum_{(x,e)  \in \vec{E}} \bigg(\sum_{y\in e, y\not=x} v_i(y) \bigg) v_i(x) \label{eq:hyper_three_term}. 
\end{align}
For the first term in \eqref{eq:hyper_three_term},
\begin{align}\label{eq:hyper_suqare}
    &\sum_{(x,e)  \in \vec{E}} \bigg(\sum_{y\in e, y\not=x} v_i(y) \bigg)^2 \\
    =&  \sum_{(x,e)  \in \vec{E}} \bigg(\sum_{y\in e} v_i(y) \bigg)^2 - 2 \sum_{(x,e)  \in \vec{E}} \bigg(\sum_{y\in e} v_i(y) \bigg) v_i(x) + \sum_{(x,e)  \in \vec{E}} v_i(x)^2.
\end{align}
Using the fact that $v_i$ is a unit vector, we have $$\sum_{(x,e)  \in \vec{E}} v_i(x)^2 = d .$$ We then calculate the first and second terms separately in the following: \begin{align}\label{eq:hyper_suqare_1}
     \sum_{(x,e)  \in \vec{E}} \bigg(\sum_{y\in e} v_i(y) \bigg)^2 &= k \sum_{y \in V} v_i(y) \bigg( \sum_{x \in V} A_{yx} v_i(x) \bigg) + kd \sum_{y \in V} v_i(y)^2 \\
     &= k  \sum_{y \in V} v_i(y) \lambda_i v_i(y) + kd = k\lambda_i + kd,
\end{align}
and \begin{align}\label{eq:hyper_suqare_2}
    \sum_{(x,e)  \in \vec{E}} \bigg(\sum_{y\in e} v_i(y) \bigg) v_i(x)
    &= \sum_{(x,e)  \in \vec{E}} \bigg(\sum_{y\in e, y \neq x} v_i(y) \bigg) v_i(x) + \sum_{(x,e)  \in \vec{E}} v_i(x)^2 \\
   &= \sum_{x \in V} \lambda_i v_i(x)^2 + d = \lambda_i+d.
\end{align}
Together, we can use \eqref{eq:hyper_suqare_1} and \eqref{eq:hyper_suqare_2} to get the value for \eqref{eq:hyper_suqare}, that is  \begin{align}
    \sum_{(x,e)  \in \vec{E}} \bigg(\sum_{y\in e, y\not=x} v_i(y) \bigg)^2 = (k-2)\lambda_i+(k-1)d .
\end{align}
Also, we have from \eqref{eq:hyper_suqare_2},
\begin{align}
     \sum_{(x,e)  \in \vec{E}} \bigg(\sum_{y\in e, y\not=x} v_i(y) \bigg) v_i(x) &=\lambda_i.
\end{align}
From \eqref{eq:hyper_three_term}, these identities give us \begin{align}\label{eq:lowerbound_H}
    \| \tilde w_i \|_2^2& =(k-1)(d+\lambda_i)(d(k-1)-\lambda_i).
\end{align}
When $\mu_i\notin\{1,(d-1)(k-1)\}$, we have $\lambda_i\notin \{d(k-1), -d\}$ from \eqref{eq:bi_mu}. Hence, $\tilde w_i$ is a nonzero vector. 
Equations \eqref{eq:infinite_norm_H} and \eqref{eq:lowerbound_H} imply \eqref{eq:deterministic_delocalization_2}.
\end{proof}

With Lemma~\ref{lem:deterministic_H}, we are ready to prove the second statement of Theorem~\ref{thm:bi_tildeB}.
\begin{proof}[Proof of Theorem~\ref{thm:bi_tildeB} (2)]
 Let $A$ be the adjacency matrix of a $(d,k)$-random regular hypergraph. The largest eigenvalue of $A$ is $d(k-1)$ with an eigenvector $v_1=\frac{1}{\sqrt n}[1,\dots,1]^\top$.   Following Lemma \ref{lem:deterministic_H}, we consider when $ \mu_i\in\{1,(d-1)(k-1)\}$ and $ \mu_i\notin\{1,(d-1)(k-1)\}$ separately.

\textit{Case 1: When $\mu_i\in\{1,(d-1)(k-1)\}$.}
$\lambda_1=d(k-1)$ as a deterministic eigenvalue for $A$ yields two eigenvalues for $\tilde B$: $\mu_1=(d-1)(k-1)$ and $\mu_1'=1$. Let  $\tilde w_1$ be the eigenvector given by \eqref{eq:w_H} associated with $\mu_1$. One can directly check that 
      \begin{align}
          \frac{ \|\tilde w_1\|_{\infty}}{\|\tilde w_1\|_2}= \frac{1}{\sqrt{nd}}=\frac{1}{\sqrt{d}}\|v_i\|_{\infty}.
      \end{align}

\textit{Case 2: When $\mu_i\notin\{1,(d-1)(k-1)\}$.} By \cite[Theorem 11]{dumitriu2021spectra},when $d\geq k\geq 3$, we have any eigenvalues $ \lambda \neq d(k-1)$ of $A$ satisfies \begin{equation}\label{eq:hyper_lambda}
   -2 \sqrt{(k-1)(d-1)} +k-2 - \varepsilon_n  \leq \lambda  \leq 2 \sqrt{(k-1)(d-1)} +k-2 + \varepsilon_n,
\end{equation}
asymptotically almost surely with $\varepsilon_n \rightarrow 0 $ as $ n \rightarrow \infty$. With \eqref{eq:hyper_lambda} and  \eqref{eq:lowerbound_H}, we have, with high probability, \begin{align}
    \|\tilde w_i\|_2^2 &\geq (k-1)(\sqrt{d-1}-\sqrt{k-1})^2(\sqrt{(k-1)(d-1)}+1)^2+o(1) \label{eq:hyper_w2}.
\end{align}
By the spectral gap bound for any eigenvalues $\mu \neq (d-1)(k-1)$ of $ B$ from \cite[Theorem 5.5]{dumitriu2021spectra}, we have \begin{equation}\label{eq:hyper_mu}
    |\mu| \leq \sqrt{(k-1)(d-1)}+o(1)
\end{equation}
with high probability. 
{Applying Lemma~\ref{lem:deterministic_H}, together with} \eqref{eq:hyper_w2} and \eqref{eq:hyper_mu}, we have, for $ d > k \geq 3$, with high probability, 
\begin{align}
    \frac{ \|\tilde w_i\|_{\infty}}{\|\tilde w_i\|_2} &\leq \frac{ ( \sqrt{(k-1)(d-1)} +1 ) \sqrt{k-1} +o(1)}{ \big( \sqrt{d-1 } - \sqrt{k-1} \big) \big( \sqrt{(k-1)(d-1)} +1 
    \big)  } \|v_i\|_{\infty}   \\
    &\leq \frac{  \sqrt{k-1}+o(1) }{  \sqrt{d-1 } - \sqrt{k-1}  } \|v_i\|_{\infty}  \label{eq:hyper_delocalization_2}.
\end{align}
Inequality \eqref{eq:hyper_delocalization_2} also holds when $\mu_i\in\{1,(d-1)(k-1)\}$. Thus, from the two cases above, Theorem~\ref{thm:bi_tildeB} (2) holds.
\end{proof}

\section{Proof of Theorem~\ref{thm:RSBM}}\label{sec:proof_RSBM}

{The spectral gap estimate for the adjacency matrix of an RSBM is known from \cite{brito2016recovery}. We recall this result below and then prove Theorem~\ref{thm:RSBM} using the spectral relation between $A$ and $\widetilde B$ established in Lemma~\ref{lem:eiganvalues}.}

\begin{lemma}\label{lem:spectral_gap}
Let $A$ be the adjacency matrix of an $(n,d_1,d_2)$-regular stochastic block model. Then with high probability, $(d_1+d_2), (d_1-d_2)$ are two eigenvalues of $A$ with multiplicity one. Moreover, for any $\varepsilon>0$,  all eigenvalues $\lambda\notin \{d_1+d_2,d_1-d_2\}$ satisfy 
\begin{align}\label{eq:spectral_bound}
   \lim_{n\to\infty} \mathbb P \left( |\lambda|\leq 2\sqrt{d_1+d_2-1}+\varepsilon\right)= 1.
\end{align}

\end{lemma}
\begin{proof}
  This is shown in  \cite[Section 4]{brito2015recovery} by using the result of the spectral gap of random lifts from \cite[Corollary 24]{bordenave2020new}, and contiguity of the permutation model and the configuration model of regular graphs from \cite[Theorem 1.3]{greenhill2002permutation}.
\end{proof}

\begin{proof}[Proof of Theorem~\ref{thm:RSBM}]
    Since an $(n,d_1,d_2)$-regular stochastic block model is a $(d_1+d_2)$-regular graph, the first statement follows from Lemma~\ref{lem:eiganvalues} directly.

    For the second statement, let $A$ be the adjacency matrix of an $(n,d_1,d_2)$ regular stochastic block model. Since $d_1-d_2$ is an eigenvalue for $A$ with associated eigenvector $\sigma$, from Lemma~\ref{lem:eiganvalues}, there are two eigenvalues of $\tilde B$ satisfying 
    \begin{align}
    \mu^2-(d_1-d_2)\mu+(d_1+d_2-1)=0.
    \end{align}
    When $(d_1-d_2)^2>4(d_1+d_2-1)$, the two eigenvalues are real, and the second statement holds.

    We now show the third statement. From Lemma~\ref{lem:spectral_gap}, since $d_1+d_2, d_1-d_2$ are two eigenvalues of $A$ with multiplicity one with high probability,  Lemma~\ref{lem:eiganvalues} shows that $d_1+d_2-1,1, \mu_2,\mu_2'$ are four real eigenvalues of $\tilde B$ with multiplicity one with high probability when $(d_1-d_2)^2>4(d_1+d_2-1)$.
Let $\lambda$ be an eigenvalue of $A$ with $\lambda\not\in \{d_1+d_2,d_1-d_2\}$. With Lemma~\ref{lem:eiganvalues}, $|\lambda|\leq 2\sqrt{d_1+d_2-1}+\varepsilon$ with high probability, and the corresponding two eigenvalues $\mu,\mu'$ of $\tilde B$ satisfies 
 \begin{align}  
    x^2-\lambda x+(d_1+d_2-1)=0.
    \end{align}
  We consider two cases:
    \begin{enumerate}
        \item If $|\lambda| \leq 2\sqrt{d_1+d_2-1}$, then the associated two eigenvalues have $|\mu|=|\mu'|=\sqrt{d_1+d_2-1}$.
        \item If $|\lambda|\in [2\sqrt{d_1+d_2-1},2\sqrt{d_1+d_2-1}+\varepsilon]$, $\mu,\mu'$ are real and 
        \begin{align}
            \mu,\mu'=\frac{\lambda\pm \sqrt{\lambda^2-4(d_1+d_2-1)}}{2}.
        \end{align}
        We have  for any $\varepsilon\in (0,1)$, with high probability,
        \begin{align}
            |\mu-\sqrt{d_1+d_2-1}|, |\mu'-\sqrt{d_1+d_2-1}|\leq C\sqrt{\varepsilon}  
        \end{align} 
        for a constant $C$ depending only on $(d_1+d_2-1)$.
    \end{enumerate}
   Therefore, all  eigenvalues $\lambda\not\in $$\{d_1+d_2-1,1,\mu_2,\mu_2' \}$ are within $o(1)$ distance from the circle of radius $\sqrt{d_1+d_2-1}$. This completes the proof of Theorem~\ref{thm:RSBM}.
\end{proof}

\section{Erd\H{o}s-R\'{e}nyi graphs}\label{sec:ER}
Studying bulk non-backtracking eigenvectors for  $G(n,p)$ is an interesting open problem. Previously, using the fact that when \( np=\omega(\log n) \) the degrees in \( G(n,p) \) are concentrated around \( np \), \cite{wang2017limiting,coste2021eigenvalues} studied the eigenvalues of \( \tilde{B} \) by approximating the degree matrix $D$ with $(np)I_n$, and analyzed a perturbed quadratic equation from the Ihara-Bass formula. However, it is not clear to us how to generalize this partial de-randomization argument introduced in \cite{wang2017limiting} to study eigenvector delocalization. We conjecture that non-backtracking eigenvector delocalization holds when \( np=\omega(\log n) \)  and does not hold for \( np=o(\log n) \), similar to the results established for the adjacency matrix \cite{tran2013sparse,dumitriu2019sparse,he2019local,alt2022completely,alt2021delocalization,alt2023poisson}. See Figure~\ref{fig:conjecture} for simulation results.

\begin{figure} 
    \centering
    \includegraphics[width=\linewidth]{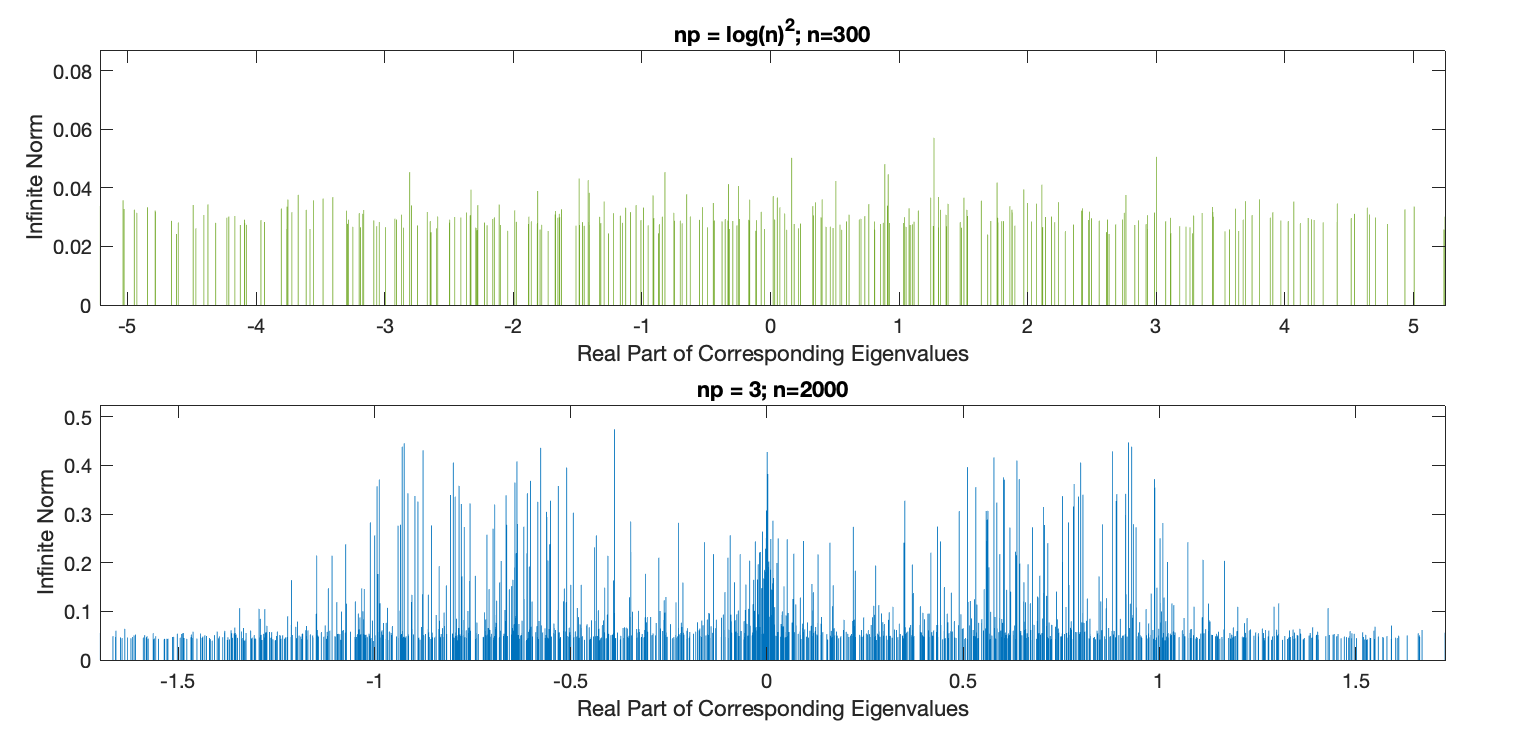}
    \vspace{-15pt}
    \caption{Simulation for Erd\H{o}s-R\'{e}nyi graphs. The upper graph is a simulation for $np = \log^2 n$ with $n = 300$; the lower graph is a simulation for $np = 3$ with $n = 2000$.} 
    \label{fig:conjecture} 
\end{figure}

\section*{Declaration of competing interest}

The authors declare that they have no known competing financial interests
or personal relationships that could have appeared to influence the work
reported in this paper.

\section*{Acknowledgments} Y.Z. thanks Paolo
Barucca, Ludovic Stephan, and Ke Wang for helpful discussions. X.Z. was supported by an AMS Undergraduate Travel Grant to present this work at Joint Mathematics Meetings 2024. Y.Z. was partially supported by the NSF-Simons Research Collaborations on the Mathematical and Scientific Foundations of Deep Learning, an AMS-Simons Travel Grant,  the Simons Grant MPS-TSM-00013944, and NSF DMS-2606337. Part of this work was completed during Y.Z.'s visit to Université Paris Cité, for which Y.Z. thanks Simon Coste for his warm hospitality.

\bibliographystyle{plain}
\bibliography{ref}

\begin{thebibliography}{10}

\bibitem{abbe2018community}
Emmanuel Abbe.
\newblock Community detection and stochastic block models: Recent developments.
\newblock {\em Journal of Machine Learning Research}, 18(177):1--86, 2018.

\bibitem{alt2021delocalization}
Johannes Alt, Raphael Ducatez, and Antti Knowles.
\newblock Delocalization transition for critical {E}rd{\H{o}}s--{R}{\'e}nyi
  graphs.
\newblock {\em Communications in Mathematical Physics}, 388(1):507--579, 2021.

\bibitem{alt2022completely}
Johannes Alt, Rapha{\"e}l Ducatez, and Antti Knowles.
\newblock The completely delocalized region of the {Erd{\H{o}}s-R{\'e}nyi}
  graph.
\newblock {\em Electronic Communications in Probability}, 27:1--9, 2022.

\bibitem{alt2023poisson}
Johannes Alt, Raphael Ducatez, and Antti Knowles.
\newblock Poisson statistics and localization at the spectral edge of sparse
  erd{\H{o}}s--r{\'e}nyi graphs.
\newblock {\em The Annals of Probability}, 51(1):277--358, 2023.

\bibitem{angel2015non}
Omer Angel, Joel Friedman, and Shlomo Hoory.
\newblock The non-backtracking spectrum of the universal cover of a graph.
\newblock {\em Transactions of the American Mathematical Society},
  367(6):4287--4318, 2015.

\bibitem{angelini2015spectral}
Maria~Chiara Angelini, Francesco Caltagirone, Florent Krzakala, and Lenka
  Zdeborov{\'a}.
\newblock Spectral detection on sparse hypergraphs.
\newblock In {\em 2015 53rd Annual Allerton Conference on Communication,
  Control, and Computing (Allerton)}, pages 66--73. IEEE, 2015.

\bibitem{barucca2017spectral}
Paolo Barucca.
\newblock Spectral partitioning in equitable graphs.
\newblock {\em Physical Review E}, 95(6):062310, 2017.

\bibitem{bass_iharaselberg_1992}
Hyman Bass.
\newblock The {Ihara}-{Selberg Zeta function} of a {tree lattice}.
\newblock {\em International Journal of Mathematics}, 03(06):717--797, December
  1992.

\bibitem{bauerschmidt2019local}
Roland Bauerschmidt, Jiaoyang Huang, and Horng-Tzer Yau.
\newblock Local {Kesten--McKay} law for random regular graphs.
\newblock {\em Communications in Mathematical Physics}, 369:523--636, 2019.

\bibitem{bauerschmidt2017local}
Roland Bauerschmidt, Antti Knowles, and Horng-Tzer Yau.
\newblock Local semicircle law for random regular graphs.
\newblock {\em Communications on Pure and Applied Mathematics},
  70(10):1898--1960, 2017.

\bibitem{benaych2020spectral}
Florent Benaych-Georges, Charles Bordenave, and Antti Knowles.
\newblock Spectral radii of sparse random matrices.
\newblock {\em Annales de l'Institut Henri Poincar{\'e}, Probabilit{\'e}s et
  Statistiques}, 56(3):2141--2161, 2020.

\bibitem{bordenave2020new}
Charles Bordenave.
\newblock A new proof of {F}riedman's second eigenvalue theorem and its
  extension to random lifts.
\newblock {\em Annales Scientifiques de l'{\'E}cole Normale Sup{\'e}rieure},
  4(6):1393--1439, 2020.

\bibitem{bordenave2012around}
Charles Bordenave and Djalil Chafai.
\newblock Around the circular law.
\newblock {\em Probability Surveys}, 9:1--89, 2012.

\bibitem{bordenave2020detection}
Charles Bordenave, Simon Coste, and Raj~Rao Nadakuditi.
\newblock Detection thresholds in very sparse matrix completion.
\newblock {\em Foundations of Computational Mathematics}, 23(5):1619--1743,
  2023.

\bibitem{bordenave2018nonbacktracking}
Charles Bordenave, Marc Lelarge, and Laurent Massouli{\'e}.
\newblock Nonbacktracking spectrum of random graphs: Community detection and
  nonregular ramanujan graphs.
\newblock {\em Annals of probability}, 46(1):1--71, 2018.

\bibitem{brito2015recovery}
Gerandy Brito, Ioana Dumitriu, Shirshendu Ganguly, Christopher Hoffman, and
  Linh~V Tran.
\newblock Recovery and rigidity in a regular stochastic block model.
\newblock {\em arXiv preprint arXiv:1507.00930}, 2015.

\bibitem{brito2016recovery}
Gerandy Brito, Ioana Dumitriu, Shirshendu Ganguly, Christopher Hoffman, and
  Linh~V Tran.
\newblock Recovery and rigidity in a regular stochastic block model.
\newblock In {\em Proceedings of the twenty-seventh annual ACM-SIAM symposium
  on Discrete algorithms}, pages 1589--1601. Society for Industrial and Applied
  Mathematics, 2016.

\bibitem{brito2021spectral}
Gerandy Brito, Ioana Dumitriu, and Kameron~Decker Harris.
\newblock Spectral gap in random bipartite biregular graphs and applications.
\newblock {\em Combinatorics, Probability and Computing}, 31(2):229--267, 2022.

\bibitem{chang2023upper}
Evan Chang, Neel Kolhe, and Youngtak Sohn.
\newblock Upper bounds on the $2 $-colorability threshold of random $ d
  $-regular $ k $-uniform hypergraphs for $ k\geq 3$.
\newblock {\em arXiv preprint arXiv:2308.02075}, 2023.

\bibitem{chodrow2023nonbacktracking}
Philip Chodrow, Nicole Eikmeier, and Jamie Haddock.
\newblock Nonbacktracking spectral clustering of nonuniform hypergraphs.
\newblock {\em SIAM Journal on Mathematics of Data Science}, 5(2):251--279,
  2023.

\bibitem{cooper1996perfect}
Colin Cooper, Alan Frieze, Michael Molloy, and Bruce Reed.
\newblock Perfect matchings in random r-regular, s-uniform hypergraphs.
\newblock {\em Combinatorics, Probability and Computing}, 5(1):1--14, 1996.

\bibitem{coste2021eigenvalues}
Simon Coste and Yizhe Zhu.
\newblock Eigenvalues of the non-backtracking operator detached from the bulk.
\newblock {\em Random Matrices: Theory and Applications}, 10(03):2150028, 2021.

\bibitem{dall2019revisiting}
Lorenzo Dall'Amico, Romain Couillet, and Nicolas Tremblay.
\newblock Revisiting the {B}ethe-{H}essian: Improved community detection in
  sparse heterogeneous graphs.
\newblock {\em Advances in neural information processing systems}, 32, 2019.

\bibitem{dumitriu2012sparse}
Ioana Dumitriu and Soumik Pal.
\newblock Sparse regular random graphs: Spectral density and eigenvectors.
\newblock {\em Annals of Probability}, 40(5):2197--2235, 2012.

\bibitem{dumitriu2019sparse}
Ioana Dumitriu and Yizhe Zhu.
\newblock Sparse general {W}igner-type matrices: Local law and eigenvector
  delocalization.
\newblock {\em Journal of Mathematical Physics}, 60(2), 2019.

\bibitem{dumitriu2021spectra}
Ioana Dumitriu and Yizhe Zhu.
\newblock Spectra of random regular hypergraphs.
\newblock {\em Electronic Journal of Combinatorics}, 28(3):P3.36, 2021.

\bibitem{dumitriu2022extreme}
Ioana Dumitriu and Yizhe Zhu.
\newblock Extreme singular values of inhomogeneous sparse random rectangular
  matrices.
\newblock {\em Bernoulli}, 30(4):2904--2931, 2024.

\bibitem{feng1996spectra}
Keqin Feng and Wen-Ching~Winnie Li.
\newblock Spectra of hypergraphs and applications.
\newblock {\em Journal of number theory}, 60(1):1--22, 1996.

\bibitem{friedman2014relativized}
Joel Friedman and David-Emmanuel Kohler.
\newblock The relativized second eigenvalue conjecture of alon.
\newblock {\em arXiv preprint arXiv:1403.3462}, 2014.

\bibitem{glover2021some}
Cory Glover and Mark Kempton.
\newblock Some spectral properties of the non-backtracking matrix of a graph.
\newblock {\em Linear Algebra and its Applications}, 618:37--57, 2021.

\bibitem{greenhill2022spanning}
Catherine Greenhill, Mikhail Isaev, and Gary Liang.
\newblock Spanning trees in random regular uniform hypergraphs.
\newblock {\em Combinatorics, Probability and Computing}, 31(1):29--53, 2022.

\bibitem{greenhill2002permutation}
Catherine Greenhill, Svante Janson, Jeong~Han Kim, and Nicholas~C Wormald.
\newblock Permutation pseudographs and contiguity.
\newblock {\em Combinatorics, Probability and Computing}, 11(3):273--298, 2002.

\bibitem{he2023edge}
Yukun He.
\newblock Edge universality of sparse erd{\H{o}}s-r{\'e}nyi digraphs: Y. he.
\newblock {\em Probability Theory and Related Fields}, pages 1--53, 2025.

\bibitem{he2019local}
Yukun He, Antti Knowles, and Matteo Marcozzi.
\newblock Local law and complete eigenvector delocalization for supercritical
  {Erd{\H{o}}s--R{\'e}nyi} graphs.
\newblock {\em Annals of Probability}, 47(5):3278--3302, 2019.

\bibitem{huang2021spectrum}
Jiaoyang Huang and Horng-Tzer Yau.
\newblock Spectrum of random $d$-regular graphs up to the edge.
\newblock {\em Communications on Pure and Applied Mathematics},
  77(3):1635--1723, 2024.

\bibitem{jost2023spectral}
J{\"u}rgen Jost, Raffaella Mulas, and Leo Torres.
\newblock Spectral theory of the non-backtracking laplacian for graphs.
\newblock {\em Discrete Mathematics}, 346(10):113536, 2023.

\bibitem{kempton2016non}
Mark Kempton.
\newblock Non-backtracking random walks and a weighted {I}hara’s theorem.
\newblock {\em Open Journal of Discrete Mathematics}, 6(4):207--226, 2016.

\bibitem{krzakala2013spectral}
Florent Krzakala, Cristopher Moore, Elchanan Mossel, Joe Neeman, Allan Sly,
  Lenka Zdeborov{\'a}, and Pan Zhang.
\newblock Spectral redemption in clustering sparse networks.
\newblock {\em Proceedings of the National Academy of Sciences},
  110(52):20935--20940, 2013.

\bibitem{li2004ramanujan}
Wen-Ching~Winnie Li.
\newblock Ramanujan hypergraphs.
\newblock {\em Geometric \& Functional Analysis GAFA}, 14(2):380--399, 2004.

\bibitem{litvak2019structure}
Alexander Litvak, Anna Lytova, Konstantin Tikhomirov, Nicole
  Tomczak-Jaegermann, and Pierre Youssef.
\newblock Structure of eigenvectors of random regular digraphs.
\newblock {\em Transactions of the American Mathematical Society},
  371(11):8097--8172, 2019.

\bibitem{lubetzky2016cutoff}
Eyal Lubetzky and Yuval Peres.
\newblock Cutoff on all {R}amanujan graphs.
\newblock {\em Geometric and Functional Analysis}, 26(4):1190--1216, 2016.

\bibitem{MCKAY1981203}
Brendan~D. McKay.
\newblock The expected eigenvalue distribution of a large regular graph.
\newblock {\em Linear Algebra and its Applications}, 40:203--216, 1981.

\bibitem{mossel2015reconstruction}
Elchanan Mossel, Joe Neeman, and Allan Sly.
\newblock Reconstruction and estimation in the planted partition model.
\newblock {\em Probability Theory and Related Fields}, 162(3-4):431--461, 2015.

\bibitem{mulas2024there}
Raffaella Mulas, Dong Zhang, and Giulio Zucal.
\newblock There is no going back: Properties of the non-backtracking laplacian.
\newblock {\em Linear Algebra and its Applications}, 680:341--370, 2024.

\bibitem{newman2014equitable}
MEJ Newman and Travis Martin.
\newblock Equitable random graphs.
\newblock {\em Physical Review E}, 90(5):052824, 2014.

\bibitem{sen2018optimization}
Subhabrata Sen.
\newblock Optimization on sparse random hypergraphs and spin glasses.
\newblock {\em Random Structures \& Algorithms}, 53(3):504--536, 2018.

\bibitem{sole1996spectra}
Patrick Sol{\'e}.
\newblock Spectra of regular graphs and hypergraphs and orthogonal polynomials.
\newblock {\em European Journal of Combinatorics}, 17(5):461--477, 1996.

\bibitem{stephan2019robustness}
Ludovic Stephan and Laurent Massouli{\'e}.
\newblock Robustness of spectral methods for community detection.
\newblock In {\em Conference on Learning Theory}, pages 2831--2860. PMLR, 2019.

\bibitem{stephan2020non}
Ludovic Stephan and Laurent Massouli{\'e}.
\newblock Non-backtracking spectra of weighted inhomogeneous random graphs.
\newblock {\em Mathematical Statistics and Learning}, 5(3):201--271, 2022.

\bibitem{stephan2022sparse}
Ludovic Stephan and Yizhe Zhu.
\newblock Sparse random hypergraphs: Non-backtracking spectra and community
  detection.
\newblock In {\em 2022 IEEE 63rd Annual Symposium on Foundations of Computer
  Science (FOCS)}, pages 567--575. IEEE, 2022.

\bibitem{stephan2023non}
Ludovic Stephan and Yizhe Zhu.
\newblock A non-backtracking method for long matrix and tensor completion.
\newblock In {\em The Thirty Seventh Annual Conference on Learning Theory},
  pages 4636--4690. PMLR, 2024.

\bibitem{storm2006zeta}
Christopher~K Storm.
\newblock The {Z}eta function of a hypergraph.
\newblock {\em The Electronic Journal of Combinatorics}, pages R84--R84, 2006.

\bibitem{terras2010zeta}
Audrey Terras.
\newblock {\em Zeta functions of graphs: A stroll through the garden}, volume
  128.
\newblock Cambridge University Press, 2010.

\bibitem{torres2021nonbacktracking}
Leo Torres, Kevin~S Chan, Hanghang Tong, and Tina Eliassi-Rad.
\newblock Nonbacktracking eigenvalues under node removal: X-centrality and
  targeted immunization.
\newblock {\em SIAM Journal on Mathematics of Data Science}, 3(2):656--675,
  2021.

\bibitem{tran2013sparse}
Linh~V Tran, Van~H Vu, and Ke~Wang.
\newblock Sparse random graphs: Eigenvalues and eigenvectors.
\newblock {\em Random Structures \& Algorithms}, 42(1):110--134, 2013.

\bibitem{wang2017limiting}
Ke~Wang and Philip~Matchett Wood.
\newblock Limiting empirical spectral distribution for the non-backtracking
  matrix of an erd{\H{o}}s-r{\'e}nyi random graph.
\newblock {\em Combinatorics, Probability and Computing}, 32(6):956--973, 2023.

\bibitem{yang2017bulk}
Kevin Yang.
\newblock Bulk eigenvalue correlation statistics of random biregular bipartite
  graphs.
\newblock {\em arXiv preprint arXiv:1705.00083}, 2017.

\end{thebibliography}

\end{document}